\pdfoutput=1
\documentclass[11pt, letterpaper]{article}
\usepackage{blindtext}
\usepackage{hyperref}
\hypersetup{
	colorlinks=true,
	linkcolor=blue!70!black,
	citecolor=blue!70!black,
	urlcolor=blue!70!black
}

\newcommand\blfootnote[1]{%
  \begingroup
  \renewcommand\thefootnote{}\footnote{#1}%
  \addtocounter{footnote}{-1}%
  \endgroup
}
\usepackage[english]{babel}
\usepackage[T1]{fontenc}
\usepackage{parskip}
\usepackage[margin=1in]{geometry}

\usepackage{amsmath}
\usepackage{todonotes}
\usepackage{amssymb}
\usepackage{amsthm}
\usepackage{booktabs}
\usepackage{accents}
\usepackage{algorithm}
\usepackage[lined,boxed,ruled,norelsize,algo2e,linesnumbered]{algorithm2e}
\usepackage{bbm}
\usepackage{bbold}
\usepackage{bm}
\usepackage{graphicx}
\graphicspath{{./figs/}}

\usepackage{cleveref}
\usepackage{thm-restate}
\usepackage{algorithm}
\usepackage{algpseudocode}

\newtheorem{theorem}{Theorem}
\newtheorem{lemma}{Lemma}
\newtheorem{fact}{Fact}
\newtheorem{definition}{Definition}
\newtheorem{corollary}{Corollary}
\newtheorem{proposition}{Proposition}

\newtheorem{remark}{Remark}

\newcommand{\defeq}{:=}
\newcommand{\norm}[1]{\left\lVert#1\right\rVert}
\newcommand{\norms}[1]{\lVert#1\rVert}
\newcommand{\normop}[1]{\left\lVert#1\right\rVert_{\textup{op}}}
\newcommand{\normf}[1]{\left\lVert#1\right\rVert_{\textup{F}}}

\newcommand{\normsop}[1]{\lVert#1\rVert_{\textup{op}}}
\newcommand{\normsf}[1]{\lVert#1\rVert_{\textup{F}}}

\newcommand{\inprod}[2]{\left\langle#1, #2\right\rangle}

\newcommand{\eps}{\epsilon}
\newcommand{\lam}{\lambda}
\newcommand{\0}{\mathbb{0}}

\newcommand{\R}{\mathbb{R}}

\newcommand{\N}{\mathbb{N}}
\newcommand{\bas}[1]{\begin{align*}#1\end{align*}}
\newcommand{\ba}[1]{\begin{align}#1\end{align}}
\newcommand{\bbb}[1]{\left[#1\right]}
\newcommand{\diag}[1]{\textbf{\textup{diag}}\left(#1\right)}
\newcommand{\half}{\frac{1}{2}}

\newcommand{\bk}{\color{black}}
\newcommand{\rd}{\color{red}}
\newcommand{\1}{\mathbbm{1}}
\newcommand{\E}{\mathbb{E}}

\newcommand{\Nor}{\mathcal{N}}

\newcommand{\Tr}{\textup{Tr}}

\newcommand{\ma}{\mathbf{A}}
\newcommand{\mw}{\mathbf{W}}

\newcommand{\id}{\mathbf{I}}
\newcommand{\bb}[1]{\left(#1\right)}

\usepackage{xcolor}
\definecolor{burntorange}{rgb}{0.8, 0.33, 0.0}

\newcommand{\rank}{\textup{rank}}

\newcommand{\Par}[1]{\left(#1\right)}
\newcommand{\Brack}[1]{\left[#1\right]}
\newcommand{\Brace}[1]{\left\{#1\right\}}
\newcommand{\Abs}[1]{\left|#1\right|}
\newcommand{\Pars}[1]{(#1)}
\newcommand{\Bracks}[1]{[#1]}

\newcommand{\oracle}{\mathcal{O}}
\newcommand{\alg}{\mathcal{A}}
\newcommand{\mproj}{\mathbf{P}}
\newcommand{\mm}{\mathbf{M}}
\newcommand{\tmm}{\widetilde{\mathbf{M}}}

\newcommand{\mb}{\mathbf{B}}
\newcommand{\mv}{\mathbf{V}}

\newcommand{\msig}{\boldsymbol{\Sigma}}

\newcommand{\tmsig}{\widetilde{\msig}}

\newcommand{\mlam}{\boldsymbol{\Lambda}}

\newcommand{\tmlam}{\widetilde{\mlam}}

\newcommand{\ms}{\mathbf{S}}

\newcommand{\ml}{\mathbf{L}}

\newcommand{\tms}{\widetilde{\ms}}
\newcommand{\tml}{\widetilde{\ml}}
\newcommand{\hml}{\widehat{\ml}}
\newcommand{\hms}{\widehat{\ms}}

\newcommand{\mmu}{\mathbf{U}}

\newcommand{\mq}{\mathbf{Q}}
\newcommand{\mzero}{\mathbf{0}}
\newcommand{\Sym}{\mathbb{S}}
\newcommand{\PSD}{\Sym_{\succeq \mzero}}
\newcommand{\PD}{\Sym_{\succ \mzero}}
\DeclareMathOperator{\linspan}{span}
\newcommand{\mn}{\mathbf{N}}
\newcommand{\orepca}{\oracle_{\textup{ePCA}}}
\newcommand{\orcpca}{\oracle_{\textup{cPCA}}}
\newcommand{\oopca}{\oracle_{1\textup{PCA}}}
\newcommand{\gam}{\gamma}
\newcommand{\Gam}{\Gamma}
\newcommand{\BBPCA}{\mathsf{BlackBoxPCA}}
\newcommand{\codeStyle}[1]{{\bfseries #1} }
\newcommand{\codeInput}{\codeStyle{Input:}}	
	
\newcommand{\codeReturn}{\codeStyle{Return:}}	
	
\newcommand{\md}{\mathbf{D}}
\newcommand{\inner}{\inprod}

\newcommand{\poly}{\textup{poly}}
\newcommand{\Prefix}{\mathsf{Prefix}}

\newcommand{\bDelta}{\overline{\Delta}}
\newcommand{\bGamma}{\overline{\Gamma}}

\newcommand{\gapind}{\mathcal{G}}
\newcommand{\dist}{\mathcal{D}}
\newcommand{\calP}{\mathcal{P}}
\newcommand{\polylog}{\textup{polylog}}
\newcommand{\hmsig}{\widehat{\mathbf{\Sigma}}}
\newcommand{\Oja}{\mathsf{Oja}}

\title{Black-Box $k$-to-$1$-PCA Reductions: Theory and Applications}
\author{Arun Jambulapati\thanks{University of Michigan, \texttt{jmblpati@gmail.com}.}\\
	\and
        Syamantak Kumar\thanks{University of Texas at Austin, \texttt{syamantak@utexas.edu}}\\
        \and
	Jerry Li\thanks{Microsoft Research, \texttt{jerrl@microsoft.com}} \\
        \and
        Shourya Pandey\thanks{University of Texas at Austin, \texttt{shouryap@utexas.edu}}\\
        \and
        Ankit Pensia\thanks{IBM Research, \texttt{ankitp@ibm.com}}\\
	\and
	Kevin Tian\thanks{University of Texas at Austin, \texttt{kjtian@cs.utexas.edu}}\\
}
\date{}
\begin{document}

\maketitle

\begin{abstract}
The $k$-principal component analysis ($k$-PCA) problem is a fundamental algorithmic primitive that is widely-used in data analysis and dimensionality reduction applications. In statistical settings, the goal of $k$-PCA is to identify a top eigenspace of the covariance matrix of a distribution, which we only have black-box access to via samples. Motivated by these  settings, we analyze black-box deflation methods as a framework for designing $k$-PCA algorithms, where we model access to the unknown target matrix via a black-box $1$-PCA oracle which returns an approximate top eigenvector, under two popular notions of approximation. Despite being arguably the most natural reduction-based approach to $k$-PCA algorithm design, such black-box methods, which recursively call a $1$-PCA oracle $k$ times, were previously poorly-understood. 

Our main contribution is significantly sharper bounds on the approximation parameter degradation of deflation methods for $k$-PCA. For a quadratic form notion of approximation we term \emph{ePCA (energy PCA)}, we show deflation methods suffer no parameter loss. For an alternative well-studied approximation notion we term \emph{cPCA (correlation PCA)}, we tightly characterize the parameter regimes where deflation methods are feasible. Moreover, we show that in all feasible regimes, $k$-cPCA deflation algorithms suffer no asymptotic parameter loss for any constant $k$. We apply our framework to obtain state-of-the-art $k$-PCA algorithms robust to dataset contamination, improving prior work in sample complexity by a $\poly(k)$ factor.\blfootnote{\textit{Accepted for presentation at the
Conference on Learning Theory (COLT) 2024}}
\end{abstract}

\newpage
\tableofcontents
\newpage

\section{Introduction}\label{sec:intro}

Principal component analysis (PCA) stands as a ubiquitous technique in the areas of statistical analysis and dimensionality reduction (see e.g.\ \cite{pearson-pca} and the popular reference \cite{joliffe-pca}), offering a powerful and general-purpose means to extract information from complex datasets. In one of its most common use cases, applied to a distribution $\dist$, the $k$-PCA problem (Definition~\ref{def:exact_pca}) seeks to identify $k$ orthogonal \emph{principal components} (PCs) which capture the most variation in $\dist$. Succinctly put, the goal is to output the top-$k$ eigenvectors of the covariance $\msig$ of $\dist$,\footnote{The standard definition of an exact $k$-PCA allows for arbitrary tiebreaking if the dimensionality of the top-$k$ eigenspace is $> k$, see Definition~\ref{def:exact_pca} for a formal statement.} or an approximation thereof, where the key challenge is that we only have black-box sample access to $\msig$.

This statistical $k$-PCA problem is extremely well-studied, and due to its fundamental nature, many recent works have explored the development of $k$-PCA algorithms under various constraints. These variations include performing PCA in small space~\cite{jain2016streaming}, under adversarial corruptions~\cite{JamLT20, DKPP23}, when the distribution is heavy-tailed~\cite{minsker2015geometric, wei2017estimation}, or when the sampling process admits correlations~\cite{neeman2023concentration, kumar2023streaming}. In some problem settings, additional constraints might be imposed on the data distribution (e.g.\ sparsity~\cite{zou2006sparse}) or the output of the PCA algorithm (e.g.\ differential privacy~\cite{liu2022dp}, fairness~\cite{lee2023fair}). In all these scenarios, asking for an exact $k$-PCA is impossible due to our black-box access to the covariance matrix $\msig$ in question, so the goal is to output an approximate $k$-PCA, for appropriate notions of approximation.

Unfortunately, in many of the aforementioned settings, such as \cite{JamLT20, DKPP23, kumar2023streaming, liu2022dp}, results are only currently known for extracting a single approximate PC, rather than the more general approximate $k$-PCA problem. However, for most practical applications of PCA such as dimensionality reduction, extracting a single PC is not sufficient and hence these algorithms do not readily apply. Given the sophisticated techniques required to prove correctness of even extracting a single PC in these constrained statistical settings, it is often not immediately clear how to extend the analysis of these works to design a corresponding $k$-PCA algorithm.

One of the most commonly-proposed solutions to alleviate the burden of $k$-PCA algorithm design, assuming existence of a corresponding $1$-PCA algorithm, is a reduction-based approach, known as a \emph{deflation method} \cite{mackey2008deflation, allen2016lazysvd} (see Algorithm~\ref{alg:bbpca}). In this framework, the extraction of a single approximate PC is treated as a subroutine, and the deflation method repeatedly projects out the results of calls to this subroutine from the dataset, $k$ times in total. It is straightforward to see that if the subroutine returned an exact top eigenvector each time, this strategy would indeed succeed in outputting an exact $k$-PCA of the target matrix. As mentioned previously, under black-box access, the key challenge is to quantify the approximation degradation of this recursive procedure. 

Indeed, despite their being perhaps the most natural reduction-based $k$-PCA strategy, there has been surprisingly-limited work rigorously analyzing the performance of deflation methods. For example, standard analyses require strong gap assumptions on the target matrix (e.g.\ each of the top-$k$ eigenvalues being separated) \cite{LiZ15}, and this condition has even been quoted in a variety of works as being necessary \cite{LiZ15, MuscoM15, shamir2016pca, LiuKO22}. The challenge of obtaining provable guarantees for deflation methods (or small-block variants thereof) was an open problem stated in \cite{MuscoM15}. Our primary contribution is a direct analysis of the approximation parameter degradation of deflation methods, with no gap assumptions on the target matrix. This refutes the aforementioned conventional wisdom that deflation methods fail without strong spectral assumptions.

Our work is motivated by a closely-related result of \cite{allen2016lazysvd}, which primarily focused on applications in the \emph{white-box} setting where the target matrix is explicitly available.\footnote{For example, in this white-box setting, one can query matrix-vector products with the target matrix, which is not a realistic assumption if the target matrix is the covariance of a distribution we have sample access to.} However, the \cite{allen2016lazysvd} analysis loses large polynomial factors in the approximation quality parameters, under their notion of approximation (Definition~\ref{def:cpca}), which we will shortly compare to our results. Our focus in this work, in comparison, is to give tighter characterizations of when it is possible to obtain lossless or nearly-lossless reductions in the black-box PCA setting, typically of interest in statistical applications.
\subsection{Our results}\label{ssec:results}

We model black-box access to a target matrix $\mm \in \PSD^{d \times d}$ through the template in Algorithm~\ref{alg:bbpca}, a black-box $k$-to-$1$-PCA reduction.\footnote{For notation used throughout the paper, see Section~\ref{sec:prelims}; $\PSD^{d \times d}$ is the set of positive semidefinite $d \times d$ matrices.} Specifically, the algorithm only interacts with $\mm$ through an oracle $\oopca$, which we assume returns an approximate top eigenvector of $\mproj \mm \mproj$ for a specified projection matrix $\mproj$. This also explains why Algorithm~\ref{alg:bbpca} is often called a deflation method, as it repeatedly ``deflates'' directions from consideration as specified by recursively calling $\oopca$.

\begin{algorithm2e}[H]
	\caption{$\BBPCA(\mm, k, \oopca)$}
	\label{alg:bbpca}
	\DontPrintSemicolon
		\codeInput $\mm \in \PSD^{d \times d}$, $k \in [d]$, $\oopca$, an algorithm which takes as input $\mm \in \PSD^{d \times d}$ and a $d \times d$ orthogonal projection matrix $\mproj$ and returns a unit vector in $\R^d$ in $\linspan(\mproj)$;
            $\mproj_0 \gets \id_d$\;
		\For{$i \in [k]$}{
            $u_i \gets \oopca(\mm, \mproj_{i - 1})$\;\label{line:oopca}
            $\mproj_i \gets \mproj_{i - 1} - u_iu_i^\top$\;
		}
    \codeReturn $\mmu \gets \{u_i\}_{i \in [k]} \in \R^{d \times k}$
\end{algorithm2e}

To motivate our approximation-tolerant analysis, first consider the performance of Algorithm~\ref{alg:bbpca} when $\oopca$ returns an exact $1$-PCA of $\mproj \mm \mproj$ on inputs $(\mm, \mproj)$, using the following definition.

\begin{definition}[Exact PCA]\label{def:exact_pca}
     \label{def:exactPCA} For $\mm \in \PSD^{d \times d}$, the \emph{exact $k$-principal component analysis} ($k$-PCA) problem asks to return orthonormal $\mv \in \R^{d \times k}$ such that $\inprod{\mv\mv^\top}{\mm}$ is maximal.
\end{definition}

If $\oopca$ is assumed to be an exact $1$-PCA algorithm, then it is straightforward to show that Algorithm~\ref{alg:bbpca} is an exact $k$-PCA algorithm (indeed, this follows as a special case of our Theorem~\ref{thm:epca_reduce}). 

In Sections~\ref{sec:epca} and~\ref{sec:cpca}, our focus is analyzing the behavior of Algorithm~\ref{alg:bbpca} when $\oopca$ only has approximate $1$-PCA guarantees. 
We consider two types of approximation which are well-studied in the literature, that we call \emph{energy PCA} (or ePCA, see Definition~\ref{def:epca}) and \emph{correlation PCA} (or cPCA, see Definition~\ref{def:cpca}). We chose these definitions of approximate PCA because of their prevalence as error metrics in the PCA literature; to our knowledge, essentially all works on PCA in the statistical setting (i.e., under sample access from a distribution) give guarantees under one of these metrics, including the aforementioned works \cite{minsker2015geometric, jain2016streaming, wei2017estimation, JamLT20, liu2022dp, kumar2023streaming, DKPP23, lee2023fair}. Correspondingly, our main results concern Algorithm~\ref{alg:bbpca} when $\oopca$ is assumed to be an approximate $1$-ePCA oracle or $1$-cPCA oracle, notions formalized in Definitions~\ref{def:epca_oracle} and \ref{def:cpca_oracle}, and we want Algorithm~\ref{alg:bbpca} to respectively return a $k$-ePCA or a $k$-cPCA of $\mm$.  
\paragraph{ePCA.} The first approximation notion we consider is \emph{energy $k$-PCA}, which defines the performance of an approximate PCA algorithm by a sum of quadratic forms over components. Intuitively, this definition quantifies the amount of variance of the data captured by the returned subspace.

\begin{definition}[Energy $k$-PCA]\label{def:epca}
Let $k \in [d]$, $\mm \in \PSD^{d \times d}$, and $\eps \in [0, 1]$. We say orthonormal $\mmu \in \R^{d \times k}$ is an \emph{$\eps$-approximate energy $k$-PCA} (or, $\eps$-$k$-ePCA) of $\mm$ if 
\[\inprod{\mmu\mmu^\top}{\mm} \ge (1 - \eps)\norm{\mm}_k,\text{ where } \norm{\mm}_k \defeq \max_{\textup{orthonormal } \mv \in \R^{d \times k}} \inprod{\mv\mv^\top}{\mm}.\]
\end{definition}
Our guarantees for Algorithm~\ref{alg:bbpca} as a black-box ePCA reduction are stated in the following result.

\begin{definition}[ePCA oracle]\label{def:epca_oracle}
We say $\orepca$ is an \emph{$\eps$-approximate $1$-ePCA oracle} (or, $\eps$-$1$-ePCA oracle) if, on inputs $\mm \in \PSD^{d \times d}$, and $\mproj \in \R^{d \times d}$, where $\mproj$ is required to be an orthogonal projection matrix, $\orepca$ returns $u \in \R^d$, an $\eps$-$1$-ePCA of $\mproj \mm \mproj$, satisfying $u \in \linspan(\mproj)$.
\end{definition}

\begin{restatable}[$k$-to-$1$-ePCA reduction]{theorem}{restateepcathm}\label{thm:epca_reduce}
Let $\eps \in (0, 1)$, let $\mm \in \PSD^{d \times d}$, and let $\oopca$ be an $\eps$-$1$-ePCA oracle (Definition~\ref{def:epca_oracle}). Then, Algorithm~\ref{alg:bbpca}, when run on $\mm$, returns $\mmu \in \R^{d \times k}$, an $\eps$-$k$-ePCA of $\mm$.
\end{restatable}

Perhaps surprisingly, Theorem~\ref{thm:epca_reduce} states that there is no loss in approximation parameters via the reduction in Algorithm~\ref{alg:bbpca}, if our approximation notion is ePCA. This is optimal; for example, if $d \ge 2k$, and $\mm = \diag{\{\1_k, \0_{d - k}\}}$, then leting $\orepca$ repeatedly return $\sqrt{1 - \eps} e_i + \sqrt{\eps }e_{k + i}$ in each iteration $i \in [k]$, the result is an $\eps$-$k$-cPCA of $\mm$. Our proof of Theorem~\ref{thm:epca_reduce} in Section~\ref{sec:epca} is a simple application of Cauchy's interlacing theorem, but to our knowledge it was not previously known.

\paragraph{cPCA.} The other notion of approximation we consider is \emph{correlation $k$-PCA}, a popular definition in the literature \cite{shamir2016pca, allen2017first} which defines the performance of an approximate PCA algorithm in terms of the correlation with the small eigenspace of $\mm$, allowing for a gap in the definition of ``small.'' Let $\mv^{< \lam}(\mm)$ denote the orthonormal matrix spanning the eigenspace of $\mm$ corresponding to eigenvalues $< \lambda$. We define cPCA as follows, parameterized by a gap $\Gamma$ and a correlation $\Delta$.

\begin{definition}[Correlation $k$-PCA]\label{def:cpca}
Let $k \in [d]$, $\mm \in \PSD^{d \times d}$, $\Gam \in [0, 1]$, and $\Delta \in [0, k]$. We say orthonormal $\mmu \in \R^{d \times k}$ is a \emph{$(\Delta,\Gam)$-approximate correlation $k$-PCA} (or, $(\Delta,\Gam)$-$k$-cPCA) of $\mm$ if 
\[
\normf{\Par{\mv^{<(1 - \Gam)\lam_k(\mm)}(\mm)}^\top \mmu}^2 \le \Delta.
\]
\end{definition}
Definition~\ref{def:cpca} requires that $\mmu$ be almost entirely-contained in an eigenspace of $\mm$ corresponding to eigenvalues which are, at worst, barely outside the top-$k$ space. For example, in the gapped setting when $\lam_2(\mm) < (1 - \Gamma)\lam_1(\mm)$ and $k = 1$ (so $\mv^{<(1 - \Gamma)\lam_k(\mm)}$ is simply the subspace orthogonal to the top PC of $\mm$), Definition~\ref{def:cpca} implies recovery of the top PC in the $\sin^2$ error metric. In general, cPCA guarantees can be viewed as more geometrically explicit than ePCA counterparts (as they bound correlation with specific eigenspaces), and hence may be preferable when they are available.\footnote{However, in some black-box settings e.g.\ robust PCA (Section~\ref{ssec:robust}), certificates on specific eigenspaces are harder to achieve, and only weaker certificates such as quadratic forms have found algorithmic use (giving ePCA guarantees).}

The only analysis of the black-box deflation method for $k$-PCA (Algorithm~\ref{alg:bbpca}) we are aware of is due to \cite{allen2016lazysvd}, who gave a black-box $k$-to-$1$-cPCA reduction under the following definition.

\begin{definition}[cPCA oracle]\label{def:cpca_oracle}
We say $\orcpca$ is a \emph{$(\delta, \gamma)$-approximate}\footnote{We sometimes use lowercase Greek letters to denote parameters for the $1$-cPCA oracle, to differentiate from capital Greek letters used to parameterize overall $k$-PCA guarantees.} \emph{$1$-cPCA oracle} (or, $(\delta, \gamma)$-$1$-cPCA oracle) if, on inputs $\mm \in \PSD^{d \times d}$, and $\mproj \in \R^{d \times d}$, where $\mproj$ is required to be an orthogonal projection matrix, $\orcpca$ returns $u \in \R^{d}$, a $(\delta,\gamma)$-$1$-cPCA of $\mproj \mm \mproj$, satisfying $u \in \linspan(\mproj)$.
\end{definition}

Under Definitions~\ref{def:cpca},~\ref{def:cpca_oracle}, the main result Theorem 4.1(a) of \cite{allen2016lazysvd} can be rephrased as follows:  Algorithm~\ref{alg:bbpca} returns a $(\Delta, \Gamma)$-$k$-cPCA if $\oopca$ is assumed to be a $(\delta,\gamma)$-$1$-cPCA oracle, for
\begin{equation}\label{eq:azl_params}\gamma = \frac \Gamma 2,\; \delta = \Theta\Par{\frac{\Delta^2\Gamma^2}{k^4 \kappa_k(\mm)^2}},\text{ where } \kappa_k(\mm)\defeq \frac{\lam_1(\mm)}{\lam_k(\mm)}.\end{equation}
In the white-box setting, $(\delta, \gamma)$-$1$-cPCA runtimes typically scale polynomially in $\frac 1 \gamma$ and logarithmically in $\frac 1 \delta$ (e.g., Theorem 1, \cite{MuscoM15}), so \cite{allen2016lazysvd} focused on maintaining a low overhead in $\gamma$. However, known sample complexity lower bounds in the black-box statistical setting (e.g., Theorem 6, \cite{allen2017first}) show $\gtrsim \frac 1 {\gamma^2 \delta}$ samples are required to obtain a $(\delta, \gamma)$-$1$-cPCA, even under strong tail bounds such as sub-Gaussianity. Hence, the \cite{allen2016lazysvd} reduction in the black-box setting uses
\begin{equation}\label{eq:azl_reduction}\approx \frac{k^4 \kappa_k(\mm)^2}{\Delta^2 \Gamma^4},\end{equation}
samples for obtaining a $(\Delta, \Gamma)$-$k$-cPCA via Algorithm~\ref{alg:bbpca},
times problem-dependent factors such as the dimension, which is a $\poly(k, \Delta, \Gamma)$ factor worse than the lower bound. Our goal is to understand whether this parameter loss is inherent for black-box cPCA reductions.\footnote{Note that Theorem~\ref{thm:epca_reduce} shows no such loss is necessary, if our approximation metric is instead ePCA.} As a starting point, in Section~\ref{sec:prelims}, we prove the following simple results transferring guarantees between Definitions~\ref{def:epca} and~\ref{def:cpca}.

\begin{restatable}{lemma}{restateetoc}\label{lem:etoc}
If $\mmu \in \R^{d \times k}$ is an $\eps$-$k$-ePCA of $\mm \in \PSD^{d \times d}$, it is a $(\frac{\eps\norm{\mm}_k}{\Gamma\lam_k(\mm)}, \Gamma)$-$k$-cPCA of $\mm$ for any $\Gamma \in (0, 1)$.
\end{restatable}

\begin{restatable}{lemma}{restatectoe}\label{lem:ctoe}
If $u \in \R^{d}$ is a $(\Delta, \Gamma)$-$1$-cPCA of $\mm \in \PSD^{d \times d}$, it is a $(\Gamma + \Delta)$-$1$-ePCA of $\mm$.
\end{restatable}

By combining Lemmas~\ref{lem:etoc} and~\ref{lem:ctoe} with our black-box ePCA reduction in Theorem~\ref{thm:epca_reduce}, it is straightforward to show Algorithm~\ref{alg:bbpca} returns a $(\Delta, \Gamma)$-$k$-cPCA if $\oopca$ is a $(\delta,\gamma)$-$1$-cPCA oracle, for
\begin{equation}\label{eq:our_params}\delta = \gamma = \Theta\Par{\frac{\Delta\Gamma}{k\kappa_k(\mm)}}.\end{equation}
Our transfer lemmas and Theorem~\ref{thm:epca_reduce} therefore already yield a sample complexity scaling as
\begin{equation}\label{eq:our_reduction}\approx \frac{k^3\kappa_k(\mm)^3}{\Delta^3\Gamma^3}\end{equation}
in standard statistical settings, which is incomparable to \eqref{eq:azl_reduction} in general but already improves upon it in, e.g., the dependences on $k$ and $\Gamma$. However, both \eqref{eq:azl_reduction} and~\eqref{eq:our_reduction} yield polynomial overheads in various approximation factors. We hence ask the natural question: can we characterize when Algorithm~\ref{alg:bbpca} serves as an asymptotically \emph{lossless} cPCA reduction, even when $k$ is a constant? 

In Section~\ref{sec:cpca}, we give our main result analyzing the use of Algorithm~\ref{alg:bbpca} as a black-box $k$-to-$1$-cPCA reduction. We show that if $\Delta \cdot \kappa_k(\mm)^2 \le \Gamma^2$ for a target matrix $\mm$, then Algorithm~\ref{alg:bbpca} indeed is an asymptotically lossless reduction for $k = \Theta(1)$. More precisely, we prove the following result.
\begin{restatable}[$k$-to-$1$-cPCA reduction]{theorem}{restatecpcathm}\label{thm:final_composition}
Let $(\Delta, \Gamma) \in (0, 1)$ and $\mm \in \PSD^{d \times d}$ satisfy $\Delta \cdot \kappa_k(\mm)^2 \le \Gamma^2$, and let $\oopca$ be a $(\delta,\gamma)$-$1$-cPCA oracle (Definition~\ref{def:cpca_oracle}), where $\delta \defeq \frac{1}{k^{\Theta(\log k)}} \cdot \Delta$, $ \gamma \defeq \frac{1}{\Theta(k^3)} \cdot \Gamma$, 
for appropriate constants.
Then, Algorithm~\ref{alg:bbpca} returns $\mmu \in \R^{d \times k}$, a $(\Delta,\Gamma)$-$k$-cPCA of $\mm$.
\end{restatable}

That is, when $\Delta \cdot \kappa_k(\mm)^2 \le \Gamma^2$, it suffices to take $\oopca$ to be a $(\delta, \gamma)$-$1$-cPCA oracle, for $\delta = \Theta_k(\Delta)$, $\gamma = \Theta_k(\Gamma)$. We complement Theorem~\ref{thm:epca_reduce} with an impossibility result (see Proposition~\ref{prop:delta_gg_gamsquare_bad}), showing that if the condition $\Delta \cdot \kappa_k(\mm)^2 \le \Gamma^2$ is violated, then Algorithm~\ref{alg:bbpca} fails to have this lossless reduction property, even when $k = 2$ and $d = 3$. In conclusion, Proposition~\ref{prop:delta_gg_gamsquare_bad} and Theorem~\ref{thm:final_composition} completely classify regimes where Algorithm~\ref{alg:bbpca} acts as a lossless cPCA reduction, up to the dependence on $k$.

An immediate question in light of Theorem~\ref{thm:final_composition} is to obtain a tighter characterization of the $k$ dependence. We view our quasipolynomial $k^{\Theta(\log k)}$ factor overhead in the parameter degradation, which is e.g.\ $d^{o(1)}$ for $k = \Theta(\log d)$, to be a proof-of-concept of Algorithm~\ref{alg:bbpca}'s utility as a $k$-cPCA algorithm design template; we conjecture it can be improved to be $\poly(k)$, which we leave open.

\paragraph{Applications.} To showcase our reductions, in Section~\ref{sec:applications} we directly apply Theorem~\ref{thm:epca_reduce} to obtain new algorithms for $k$-ePCA in the statistical setting, which are robust to adversarial corruptions to the dataset. Our first such result extends a recent near-linear time $1$-PCA algorithm of \cite{DKPP23}.

\begin{restatable}[Robust sub-Gaussian $k$-ePCA]{theorem}{restaterobsubg}\label{thm:robust_subg}
    Let $\dist$ be an $O(1)$-sub-Gaussian distribution on $\R^d$ with covariance $\mathbf \Sigma$.
    Let $\eps \in (0,\eps_0)$ for an absolute constant $\eps_0$, and $\delta \in (0,1)$. 
    Let $T$ be an $\eps$-corrupted set of samples from $\dist$ with $|T| = \Theta( \frac{d + \log(1/\delta)}{\eps^2 \log^2(1/\eps)} )$ for an appropriate constant. Algorithm $\alg_k$ (Corollary~\ref{cor:k-epca-stability}) run on inputs $T$, $\eps$, $\gamma = \Theta(\eps \log(\frac 1 \eps))$, $\delta$, and $k \in [d]$ outputs orthonormal $\mathbf U \in \R^{d \times k}$ such that, with probability $ \ge 1 - \delta$, $\mathbf{U}$ is an $O(\eps\log(\frac 1 \eps))$-$k$-ePCA of $\mathbf{\Sigma}$, in time $O\Pars{\frac{ndk}{\eps^2}\polylog\Pars{\frac d {\eps\delta}}}$.
\end{restatable}

We then consider a heavy-tailed variant of the robust statistical estimation problem in Theorem~\ref{thm:robust_subg}, where rather than requiring our distribution to be sub-Gaussian, we ask that it is hypercontractive (Definition~\ref{def:hypercontractivity}), the de facto finite moment bound assumption used in the PCA literature (see e.g.\ \cite{Mendelson2018RobustCE}). Under this heavy-tailed assumption, we design a new robust $1$-ePCA algorithm following a \emph{stability} framework of \cite{DKPP23} (see Definition~\ref{def:stability}), and prove the following result.

\begin{restatable}[Robust hypercontractive $k$-ePCA]{theorem}{restaterobht}
\label{thm:robust-k-epca-heavy-tailed}
    For an even integer $p \ge 4$, let $\mathcal D$ be $(p,C_p)$-hypercontractive on $\R^d$
with mean $\0_d$ and covariance $\msig$.
    Let $\eps \in (0,\eps_0)$, $\delta \in (0,1)$, and 
    $\gamma = \Theta(C_p^2 \eps^{1 - \frac{2}{p}})$ such that $\gamma \in (0,\gamma_0)$ for absolute constants $\eps_0, \gamma_0$.
    Let $T$ be an $\eps$-corrupted set of samples from $\mathcal D$ with $|T| = \Theta(\beta (\frac{d \log d + \log(1/\delta)}{\gamma^2} ))$ for an appropriate constant, where $\beta \defeq C_p^6 \eps^{-\frac 2 p}$.
    Algorithm $\alg_k$ (Corollary~\ref{cor:k-epca-stability}) run on inputs $T$, $\eps$, $\gamma$, $\delta$, and $k \in [d]$ outputs orthonormal $\mathbf U \in \R^{d \times k}$ such that, with probability $\ge 1 - \delta$, $\mathbf{U}$ is an $O(\gamma)$-$k$-ePCA of $\mathbf{\Sigma}$, in time $O\Pars{\frac{ndk}{\gamma^2}\polylog\Pars{\frac{d}{\eps\delta}}}$.
\end{restatable}

We show our approximation factor $\gamma$ in Theorem~\ref{thm:robust-k-epca-heavy-tailed} is tight up to constants in Lemma~\ref{lem:lower-bound-robust}. Moreover, we mention that even for uncorrupted data and $k = 1$, we are unaware of any other polynomial-time PCA algorithm with a similar additive dependence on $\log(\frac 1 \delta)$ in the sample complexity under a heavy-tailed assumption (rather than sub-Gaussianity), improving on \cite{minsker2015geometric}. 

Interestingly, both Theorems~\ref{thm:robust_subg} and~\ref{thm:robust-k-epca-heavy-tailed} incur no overhead in $k$ in their sample complexities, since our stability-based approach composes with deflation (see Lemma~\ref{lem:stable_deflation}). The only other robust $k$-PCA algorithm we are aware of (under the well-studied definition of robustness in Definition~\ref{def:contaminaton-model}) is due to Proposition 2.6, \cite{KonSKO20}, which applies to distributions satisfying an assumption which is a variant of $(4, C_4)$-hypercontractivity, but is not directly comparable. We give a comparison between our two results in the case $p = 4$ at the end of Section~\ref{ssec:robust}, where we show that under the more standard assumption of hypercontractivity, our recovery rate matches the rate obtained by \cite{KonSKO20}, while our sample complexity improves by a factor of up to $k^2$.

Finally, to show how to apply Theorem~\ref{thm:final_composition} to statistical settings, we design an online $k$-cPCA algorithm for heavy-tailed distributions in Theorem~\ref{thm:oja_k} by using Oja's algorithm \cite{oja1982simplified}. While Theorem~\ref{thm:oja_k} does not match the state-of-the-art due to \cite{allen2017first}, it follows a more black-box approach rather than requiring a custom sophisticated analysis. As new $1$-cPCA algorithms emerge in the literature, e.g.\ under the aforementioned input or output constraints, we are optimistic the roadmap in Theorem~\ref{thm:oja_k} will serve as a useful template for solving the corresponding $k$-cPCA problems.
\subsection{Our techniques}\label{ssec:techniques}

We now give a high-level overview of our proofs for Theorems~\ref{thm:epca_reduce} and~\ref{thm:final_composition}. Letting 
$\mm \in \PSD^{d \times d}$, the centerpiece of both our proofs is analyzing the approximation parameter degradation of composing $\mmu_{1} \in \R^{d \times k_{1}}$, an approximate $k_1$-PCA of $\mm$, and $\mmu_2 \in \R^{d \times k_2}$, an approximate $k_2$-PCA of $\tmm \defeq (\id_d - \mmu_1\mmu_1^\top)\mm(\id_d - \mmu_1\mmu_1^\top)$ with $\mmu_2 \perp \mmu_1$, to form an approximate $(k_1 + k_2)$-PCA of $\mm$. We describe how we analyze compositions under both types of approximation we consider.

\paragraph{Energy $k$-PCA.} In the case of ePCA, we analyze this composition using Lemma~\ref{lem:best_proj}, an application of Cauchy's interlacing theorem. Interestingly, we show that the composition of an $\eps$-$k_1$-ePCA and an $\eps$-$k_2$-ePCA in the sense described previously results in an $\eps$-$(k_1 + k_2)$-ePCA of $\mm$ (Corollary~\ref{cor:epca_compose}), i.e.\ we show that ePCA composition suffers no approximation parameter degradation. We defer the details of our proof to Section~\ref{sec:epca}, owing to the conciseness of the argument.

\paragraph{Correlation $k$-PCA.} In the case of cPCA, our proof relies on Lemma~\ref{lemma:u1_u2_composition}, which bounds the quality of a composition of two cPCAs. To describe this result, we introduce some notation. Let $\mmu_1$ be a $(\delta, \gamma)$-$k_1$-cPCA of $\mm$ and $\mmu_2$ be a $(\delta, \gamma)$-$k_2$-cPCA of $\tmm$, and let $\mmu \defeq \begin{pmatrix} \mmu_1 & \mmu_2 \end{pmatrix}$ and $k \defeq k_1 + k_2$. Moreover, for parameters $(\Delta,\Gamma)$ to be set, let $\ml \defeq \mv^{\ge (1 - \Gamma)\lam_{k}(\mm)}$ be the large eigenspace of $\mm$ that $(\Delta, \Gamma)$-$k$-cPCAs are required to be correlated with, let $\mlam$ be diagonal containing eigenvalues of $\mm$ corresponding to $\ml$, and let $\tml \defeq \mv^{\ge (1 - \gamma_2)\lam_{k_2}(\tmm)}(\tmm)$ be the analogous large eigenspace of the residual matrix $\tmm$. Then,  Lemma~\ref{lemma:u1_u2_composition} states that $\mmu$ is a $(\Delta, \Gamma)$-$k$-cPCA of $\mm$, for
\ba{
   \Delta \defeq 3\delta + \frac{4\delta}{\gamma^2} \cdot \frac{\normop{\mmu_{1}^\top\ml \mlam \ml^{\top}\tml}^{2}}{\lambda_{k}\bb{\mm}^{2}}, \; \Gamma \defeq 2\gamma.\label{eq:U1_U2_composition_sketch}
}
In other words, as long as we can appropriately control the quantity $\normsop{\mmu_{1}^\top\ml \mlam \ml^{\top}\tml}$, we can bound the parameter degradation (from $(\delta, \gamma)$ to $(\Delta, \Gamma)$) of composing two cPCAs. For instance, if this quantity is $O(\gamma \lam_k(\mm))$, then cPCA composition blows up both the $\delta$ and $\gamma$ parameters by constants.

Our proof of \eqref{eq:U1_U2_composition_sketch} applies the cPCA guarantees of $\mmu_1$ and $\mmu_2$ in turn, and uses a decomposition inspired by the gap-free Wedin theorem of \cite{allen2016lazysvd} to bound a cross term arising in the analysis. In light of \eqref{eq:U1_U2_composition_sketch}, we focus on bounding $\normsop{\mmu_{1}^\top\ml \mlam \ml^{\top}\tml}$, i.e.\ for arbitrary unit-norm $u \in \linspan(\mmu_1)$ and $w \in \linspan(\tml)$, our goal is to bound $u^\top \ml \mlam \ml^\top w$, where by definition of our algorithm, $u \perp w$.
Note that by our cPCA assumption, $u$ is highly-uncorrelated with $\ms \defeq \ml_\perp$, the small eigenspace of $\mm$. 

Our first observation is that when the spectrum of $\mm$ is gapped, in that $(1 - \gamma)\lam_{k_1}(\mm) > \lam_{k_1 + 1}(\mm)$, the cPCA guarantee of $\mmu_1$ affords an additional \emph{head guarantee} (Definition~\ref{def:head_guarantee}) by taking complements. Concretely, $\ml$ is simply the top-$k_1$ eigenspace of $\mm$ in the gapped setting, and because $\mmu_1$ captures almost all of the spectral mass of $\ml$, $w$ (which is orthogonal to $\mmu_1$) is highly-uncorrelated with $\ml$. In other words, we have shown that in the gapped setting, $u$ and $w$ lie almost entirely in complementary eigenspaces of $\mm$. By exploiting this fact along with $u \perp w$, it is straightforward to bound $u^\top \ml \mlam \ml^\top w = O(\gamma \lam_k(\mm))$,\footnote{In this step, we require that $(\Delta, \Gamma)$ do not lie in the infeasible parameter regime ruled out by Proposition~\ref{prop:delta_gg_gamsquare_bad}.} as we show in Lemma~\ref{lem:gap_merge_two}, so that $\Delta = O(\delta)$ in \eqref{eq:U1_U2_composition_sketch}.

Our second observation is that whenever the spectrum of $\mm$ is not gapped, but $\mm$ has a well-conditioned top-$k$ eigenspace, we can nonetheless use our head guarantee strategy to bound \eqref{eq:U1_U2_composition_sketch} with some amount of approximation parameter degradation, shown in Lemma~\ref{lem:nogap_merge_one}. To do this, we show that in this setting, we can artificially induce a head guarantee with a loss of $\approx k$ in the gap parameter by carefully inducting on intermediate cPCA guarantees (see Lemma~\ref{lem:no_gap_head_guarantee}). This lets us conclude that while $u, w$ do not lie in entirely-complementary eigenspaces, the spectrum of $\mm$ corresponding to the overlapping eigenspace lies in a very narrow range. By plugging in our bounds into \eqref{eq:U1_U2_composition_sketch}, we show that $\Delta = \delta \cdot \poly(k)$ and $\Gamma = 2\gamma$ suffices in the well-conditioned regime.

After establishing complementary results in Lemma~\ref{lem:gap_merge_two} (handling the gapped setting) and Lemma~\ref{lem:nogap_merge_one} (handling the well-conditioned setting), we show how to bucket the top-$k$ spectrum of $\mm$ so that one of these two analyses always applies. We then recursively apply our composition results via dyadic merging strategies, leading to Theorem~\ref{thm:final_composition}, where the quasipolynomial $k^{\Theta(\log(k))}$ loss arises due to the number of merges. We consider the possibility of $\poly(k)$-overhead black-box cPCA reductions to be an interesting open direction suggested by our work, which likely requires new ideas.
\subsection{Related work}\label{ssec:related}

We are aware of few papers considering deflation-based approaches to $k$-PCA without strong gap assumptions on the target matrix (see e.g.\ \cite{LiZ15} for a gap-based analysis). The most direct comparison is \cite{allen2016lazysvd}, who analyze 
Algorithm~\ref{alg:bbpca} when $\oopca$ is a $1$-cPCA oracle, giving guarantees summarized in \eqref{eq:azl_params}. Our Theorem~\ref{thm:epca_reduce} shows the lossiness in the \cite{allen2016lazysvd} reduction is unnecessary if our approximations are instead measured through Definition~\ref{def:epca} (ePCA), and yields a competitive guarantee \eqref{eq:our_params} for cPCA. 
Our Theorem~\ref{thm:final_composition} further removes all overhead in this cPCA reduction for $k = \Theta(1)$, in the valid parameter regimes not ruled out by our lower bound,  Proposition~\ref{prop:delta_gg_gamsquare_bad}.

We also draw inspiration from the extensive body of work focusing on computationally efficient and statistically optimal algorithms for PCA, e.g.\ \cite{MuscoM15, garber2016faster, shamir2016pca} which give runtime bounds in the white-box setting, and e.g.\ \cite{HardtP14, SaRO15, BalcanDWY16, shamir2016pca, jain2016streaming, LiWLZ18} which focus on offline and online variants of statistical PCA. In particular, the gap-free notion of approximation in our cPCA definition is patterned off analogous definitions in this latter line. Further, several other notions of PCA approximation have been proposed in the white-box setting, e.g.\ Eqs.\ (1)-(3) of \cite{MuscoM15}; however, all such results we are aware of use matrix-vector products. Our motivation is analyzing Algorithm~\ref{alg:bbpca} as a reduction, so we focused on error metrics where black-box $1$-PCA guarantees already exist, though it is interesting to prove similar reductions for other metrics. 

Additionally, we mention the contemporaneous work of \cite{meyer2024unreasonable}, which gives analyses of variants of $k$-to-$1$ reductions for block Krylov methods targeting low-rank approximation, which is very close in spirit to our work. However, \cite{meyer2024unreasonable} again requires access to matrix-vector products for their reductions, which makes their results not directly applicable in our black-box setting.

Finally, iterative deflation has been popular in the context of sparse PCA~\cite{d2004direct,moghaddam2006generalized, mackey2008deflation,yuan2013truncated}, as designing direct $k$-PCA methods is challenging in this setting. However, we are not aware of any general theoretical guarantees on this technique.
\section{Preliminaries}\label{sec:prelims}

For $n \in \N$ we let $[n] \defeq \{i \in \N \mid i \le n\}$. We denote the $i^{\text{th}}$ canonical basis vector in $\R^d$ by $e_i$. 

We denote matrices in capital boldface letters. For matrices $\ma, \mb$ with the same number of rows, we let $\begin{pmatrix} \ma & \mb \end{pmatrix}$ denote the horizontal concatenation of the matrices. We let $\id_d$ be the $d \times d$ identity matrix and $\mzero_d$ be the all-zeroes matrix. We say $\mv \in \R^{d \times r}$ is orthonormal if its columns are orthonormal, i.e.\ $\mv^\top\mv = \id_r$; note that $d \ge r$ necessarily in this case. We let $\Sym^{d \times d}$ be the set of real symmetric $d \times d$ matrices, which we equip with the Loewner partial ordering $\preceq$ and the Frobenius inner product $\inprod{\mm}{\mn} = \Tr(\mm\mn)$. We let $\PSD^{d \times d}$ and $\PD^{d \times d}$ respectively denote the positive semidefinite and positive definite subsets of $\Sym^{d \times d}$. We say $\ma \in \R^{m \times n}$ has singular value decomposition (SVD) $\mmu \msig \mv^\top$ if $\ma = \mmu\msig\mv^\top$ has rank-$r$, $\msig \in \R^{r \times r}$ is diagonal, and $\mmu \in \R^{m \times r}$, $\mv \in \R^{n \times r}$ are orthonormal. We refer to the $i^{\text{th}}$ largest eigenvalue of $\mm \in \Sym^{d \times d}$ by $\lam_i(\mm)$, and the $i^{\text{th}}$ largest singular value of $\mm \in \R^{m \times n}$ by $\sigma_i(\mm)$. For $\mm \in \PSD^{d \times d}$ and $k \in [d]$ we use $\norm{\mm}_k \defeq \sum_{i \in [k]} \lam_i(\mm)$ to refer to its $k$-Ky Fan norm, and $\normop{\mm} \defeq \sigma_1(\mm)$ is the $\ell_2$ operator norm of $\mm \in \R^{m \times n}$. For $\ma \in \R^{m \times n}$, the span of $\ma$'s columns, i.e.\ $\{\ma v \mid v \in \R^n\}$, is denoted $\linspan(\ma)$. For $\mm \in \PD^{d \times d}$ and $k \in [d]$, we define
\[\kappa_k(\mm) \defeq \frac{\lam_1(\mm)}{\lam_k(\mm)},\]
to be the $k$-condition number of $\mm$, i.e.\ the condition number of the top-$k$ eigenspace of $\mm$, and we abbreviate $\kappa(\mm) \defeq \kappa_d(\mm)$ for the (standard) condition number of $\mm$.

For a subspace $V \subseteq \R^d$, we let $\mproj_V \in \R^{d \times d}$ be the orthogonal projection matrix onto $V$. We often identify a subspace $V \subseteq \R^d$ with a generating matrix $\mv \in \R^{d \times r}$, where $r \defeq \dim(V)$, such that $\mv$ has orthonormal columns and $\linspan(\mv) = V$. In this case, note that $\mproj_V = \mv\mv^\top$ for any valid $\mv$. We write $U \perp V$ for subspaces $U, V \subset \R^d$ if every element of $U$ is orthogonal to every element of $V$. Similarly, for two orthonormal matrices $\mmu, \mv$ with the same number of rows, we write $\mmu \perp \mv$ if $\linspan(\mmu) \perp \linspan(\mv)$. For orthonormal $\mmu \in \R^{d \times r}$, we write $\mmu_\perp$ to mean any orthonormal $d \times (d - r)$ matrix so that columns of $\mmu$ and $\mmu_\perp$ form a basis of $\R^d$, and define $U_\perp$ to be the orthogonal complement of a subspace $U \subseteq \R^d$. For $\mm \in \Sym^{d \times d}$ and $\lam > 0$, we denote the subspace spanned by eigenvectors of $\mm$ corresponding to eigenvalues at least $\lambda$ by $V^{\ge \lam}(\mm)$, and the corresponding orthonormal generating matrix by $\mv^{\ge \lam}(\mm)$. We similarly define $V^{>\lam}(\mm)$, $V^{\le\lam}(\mm)$, etc.\ 

\paragraph{Linear algebra preliminaries.} We use the following well-known linear algebra results.

\begin{fact}[Cauchy's interlacing theorem]\label{fact:cauchy_interlacing}
Let $d \in \N$ and $r \in [d]$.
For any $\mm \in \Sym^{d \times d}$ and orthonormal $\mv \in \R^{d \times r}$,
$\lam_{d - r + i}(\mm) \le \lam_i(\mv^\top\mm\mv) \le \lam_i(\mm)$ for all $i \in [r]$.
\end{fact}

\begin{fact}[Weyl's perturbation inequality]\label{fact:weyl}
Let $\mm, \md \in \Sym^{d \times d}$ have $\normop{\md} \le \delta$. Then,
\[\lam_j(\mm) -\delta \le \lam_j(\mm + \md) \le \lam_j(\mm) + \delta, \text{ for all } j \in [d].\]
\end{fact}
\begin{fact}
    \label{fact:trace_cs}
    For any $\mathbf{A}, \mathbf{B} \in \mathbb{R}^{m \times n}$, $2\normsf{\ma}^2 + 2\normsf{\mb}^2 \ge \normsf{\ma + \mb}^2$.
\end{fact}
\begin{fact}
    \label{fact:generalized_mean}
    For any $\{x_i\}_{i \in [n]} \subset \R_{> 0}$ and $\{w_i\}_{i \in [n]} \subset \R_{\ge 0}$ with $\sum_{i \in [n]} w_i = 1$, we have for all $p \ge 1$,
    \[\bb{\sum_{i \in [n]} w_ix_i^p}^{\frac 1 p} \ge \inprod{w}{x}.\]
\end{fact}
\begin{fact}
\label{fact:matrix-inner-prodct}
    For $\ma, \mb \in \Sym^{d \times d}$ with $\ma \preceq \mb$ and $\mm \in \PSD^{d \times d}$, $\Tr(\ma\mm) \leq \Tr(\mb\mm) $. 
\end{fact}

Finally, we prove the following basic observations relating Definitions~\ref{def:epca} and~\ref{def:cpca}, stated earlier, which allow us to convert an ePCA guarantee to a cPCA one and vice-versa. 

\restateetoc*
\begin{proof}
Define $\mv \defeq \mv^{\ge (1 - \Gam)\lam_k(\mm)}{(\mm)}$, $\mw \defeq \mv^{<(1 - \Gam)\lam_k(\mm)}{(\mm)}$, and $\Delta \defeq \normsf{\mw^\top\mmu}^2$. Note $\Tr(\mv^\top\mmu\mmu^\top\mv)$ $= k - \Delta$, and $\mv^\top\mmu\mmu^\top\mv$ has $\le k$ nonzero eigenvalues, all in $[0, 1]$. We bound
\ba{
(1 - \eps)\norm{\mm}_k &\le \inprod{\mmu\mmu^\top}{\mm} = \inprod{\mmu\mmu^\top}{\mv\mv^\top\mm\mv\mv^\top} + \inprod{\mmu\mmu^\top}{\mw\mw^\top\mm\mw\mw^\top} \notag \\
&\le \inprod{\mv^\top\mmu\mmu^\top\mv}{\mv^\top\mm\mv} + \Delta\normop{\mw^\top\mm\mw} \notag \\
&\le \norm{\mm}_k - \Delta\lam_k(\mm) + \Delta(1 - \Gamma)\lam_k(\mm) = \norm{\mm}_k - \Delta\Gamma\lam_k(\mm) \label{eq:epca_bound}.
}
The second line used \Cref{fact:matrix-inner-prodct} by noting that $\inprod{\mmu\mmu^\top}{\mw\mw^\top\mm\mw\mw^\top} = \Tr(\mw^\top\mm\mw \mw^\top  \mmu \mmu^\top \mw)$ and $\Tr(\mw^\top  \mmu \mmu^\top \mw) \leq \Delta$,
and the third used the von Neumann trace inequality, which shows that if $\{\lam_i\}_{i \in [\ell]}$ are the nonincreasing eigenvalues of $\mv^\top\mm\mv$ and $\{\sigma_i\}_{i \in [\ell]}$ are the nonincreasing eigenvalues of $\mv^\top\mmu\mmu^\top\mv$,
\[\inprod{\mv^\top\mmu\mmu^\top\mv}{\mv^\top\mm\mv} \le \sum_{i \in [\ell]}\sigma_i \lam_i.\]
Since the $\sigma_i$ sum to $k - \Delta$ and are each at most $1$, it is clear that the right-hand side above is bounded by $\norm{\mm}_k - \Delta\lam_k(\mm)$ as claimed.
The conclusion follows via rearranging \eqref{eq:epca_bound}.
\end{proof}

\restatectoe*
\begin{proof}
Define $\mv \defeq \mv^{\ge (1 - \Gam)\lam_1(\mm)}$ and $\mw \defeq \mv^{< (1 - \Gam)\lam_1(\mm)}$. The conclusion follows from
\begin{align*}
\inprod{uu^\top}{\mm} &= \inprod{uu^\top}{\mv\mv^\top\mm\mv\mv^\top} + \inprod{uu^\top}{\mw\mw^\top\mm\mw\mw^\top} \\
&\ge \inprod{\mv^\top uu^\top \mv}{\mv^\top\mm\mv} \ge \inprod{\mv^\top uu^\top \mv}{(1 - \Gam)\lam_1(\mm)\id_{\dim(\linspan(\mv))}} \\
&\ge (1 - \Delta)(1 - \Gamma)\lam_1(\mm) \ge (1 - \Delta - \Gamma)\lam_1(\mm),
\end{align*}
where $\mv\mv^\top$, $\mw\mw^\top$, and $\mm$ all commute since $\mv$ and $\mw$ are eigenvectors of $\mm \in \PSD^{d \times d}$.
\end{proof}

\begin{remark}\label{rem:ctoe_weird_k}
One interesting asymmetry between Lemma~\ref{lem:ctoe} and its counterpart, Lemma~\ref{lem:etoc}, is that there does not appear to be a natural $k$-PCA conversion analog of Lemma~\ref{lem:ctoe}. For example, consider the case where $\mmu \in \R^{d \times k}$ is a $(\Delta,\Gamma)$-$k$-cPCA of $\mm \in \PSD^{d \times d}$, but $\mm$ has $\ge k$ eigenvalues in the range $(\lam_k(\mm), (1 - \frac{\Gamma}{2})\lam_{k}(\mm))$, and $\linspan(\mmu)$ is entirely contained in the corresponding subspace. Then, $\inprod{\mmu\mmu^\top}{\mm} < k\lam_k(\mm)$, which can be arbitrarily smaller than $\norm{\mm}_k$ if $\kappa_k(\mm)$ is large.
\end{remark}

It is important to note that na\"ively applying such conversions with our lossless ePCA result in Theorem~\ref{thm:epca_reduce} do not provide us with optimal cPCA overheads, as in Theorem~\ref{thm:final_composition}, even when $k = 2$ and $\kappa_k(\mm) = 1$. Our Theorem~\ref{thm:final_composition} does not suffer from such lossiness, up to an overhead only depending on $k$, except in a parameter regime ruled out by Proposition~\ref{prop:delta_gg_gamsquare_bad}.

Before we begin describing our techniques, we define another desirable property of the output of the cPCA algorithms, which is complementary to the guarantee of Definition~\ref{def:cpca}.
\begin{definition}[Head guarantee for $k$-cPCA]\label{def:head_guarantee}
For $k \in [d]$, $\mm \in \PSD^{d \times d}$, let $\mmu \in \R^{d \times k}$ be orthonormal. We say $\mmu$ satisfies a \emph{$(h, \omega, \Delta)$-head guarantee} with respect to $\mm$ for $\omega, \Delta \in \mathbb{R}_{\geq 0}$, $0 \le h \le d$, if 
\begin{equation}\label{eq:head_conditions}
\lam_h(\mm) > \lam_{h + 1}(\mm),\;
\normf{\Par{\mv^{\ge\lam_h(\mm)}(\mm)}^\top \mmu_{\perp}}^2 \le \Delta, \text{ and } \frac{\lambda_{h+1}\bb{\mm}}{\lambda_{k}\bb{\mm}} \le 1 + \omega.
\end{equation}
If $h = 0$, we only require the third condition above.
\end{definition}
Definition~\ref{def:head_guarantee} asks that $\mmu_\perp$ is highly-uncorrelated with the top-$h$ eigenspace of $\mm$, so that $\mmu$ has picked up most of the spectral mass, for an $h$ such that $\lam_{h + 1}(\mm)$ is not very separated from $\lam_k(\mm)$.
Note that Definition~\ref{def:cpca} does not explicitly guarantee a head guarantee in the absence of a gap.
\section{$k$-to-$1$-ePCA reduction}\label{sec:epca}

In this section, we provide a guarantee on Algorithm~\ref{alg:bbpca}'s performance when $\oopca$ in Line~\ref{line:oopca} is an $\eps$-$1$-ePCA oracle. We begin by stating one helper lemma used in the analysis.

\begin{lemma}\label{lem:best_proj}
Let $\mm \in \PSD^{d \times d}$ and $r \in [d]$. If $\mproj \in \R^{d \times d}$ is a rank-$(d - r)$ orthogonal projection matrix, then $\normop{\mproj \mm \mproj} \ge \lam_{r + 1}(\mm)$.
\end{lemma}
\begin{proof}
This is a consequence of Cauchy's interlacing theorem (Fact~\ref{fact:cauchy_interlacing}).
\end{proof}

\restateepcathm*
\begin{proof}
We proceed by induction on $i \in [k]$; for disambiguation let $\mmu_i \in \R^{d \times i}$ denote the horizontal concatenation of the first $i$ calls to $\oopca$, so that $\mproj_i = \id_d - \mmu_i\mmu_i^\top$. Observe that
\begin{align*}
\Tr\Par{\mmu_{i + 1}^\top \mm \mmu_{i + 1}} &= \Tr\Par{\mmu_i^\top \mm \mmu_i} + u_{i + 1}^\top \mm u_{i + 1} \\
&\ge (1 - \eps) \norm{\mm}_i + (1 - \eps) \normop{\mproj_i \mm \mproj_i} \\
&\ge (1 - \eps) \norm{\mm}_i + (1 - \eps) \sigma_{i + 1}(\mm) = (1 - \eps) \norm{\mm}_{i + 1}.
\end{align*}
The first inequality used the inductive hypothesis (the base case is $i = 0$) and the guarantee on $u_{i + 1}$ from Definition~\ref{def:epca_oracle}, and the second used Lemma~\ref{lem:best_proj}. The conclusion follows by taking $i = k$.
\end{proof}

We mention one elegant generalization of the proof strategy of Theorem~\ref{thm:epca_reduce}, which shows that arbitrary compositions of two $\eps$-ePCAs remain an $\eps$-ePCA, regardless of the block size.

\begin{corollary}\label{cor:epca_compose}
Let $\eps \in (0, 1)$, let $\mmu_1 \in \R^{d \times k_1}$ be an $\eps$-$k_1$-ePCA of $\mm \in \PSD^{d \times d}$, and define $\tmm \defeq (\id_d - \mmu_1\mmu_1^\top) \mm (\id_d - \mmu_1\mmu_1^\top)$. Let $\mmu_2 \in \R^{d \times k_2}$ be an $\eps$-$k_2$-ePCA of $\tmm$ with $\mmu_2 \perp \mmu_1$, and let $k \defeq k_1 + k_2$. Then, $\mmu \defeq \begin{pmatrix} \mmu_1 & \mmu_2 \end{pmatrix} \in \R^{d \times k}$ is an $\eps$-$k$-ePCA of $\mm$.
\end{corollary}
\begin{proof}
The proof is identical to Theorem~\ref{thm:epca_reduce}, where we use Fact~\ref{fact:cauchy_interlacing} to conclude
\[\Tr\Par{\mmu_{2}^\top \mm \mmu_2} \ge (1 - \eps)\norm{\Par{\id_d - \mmu_1\mmu_1^\top} \mm \Par{\id_d - \mmu_1\mmu_1^\top}}_{k_2} \ge (1 - \eps)\sum_{i \in [k_2]} \lam_{k_1 + i}(\mm).\]
\end{proof}
\section{$k$-to-$1$-cPCA reduction}\label{sec:cpca}

In this section, we give our second black-box $k$-to-$1$-PCA reduction, which considers the case where $\oopca$ in Line~\ref{line:oopca} of Algorithm~\ref{alg:bbpca} is a $(\delta,\gamma)$-1-cPCA oracle. In Section~\ref{ssec:cpca_prelims} we first give two helper tools which are used throughout to bound the spectra of matrices under projections formed via perturbations of its eigenvectors. In Section~\ref{ssec:cpca_lower}, we establish an impossibility result on $k$-to-$1$-cPCA reductions in a range of the parameters $(\delta, \gamma)$. In the complement of this parameter regime, we analyze the performance of Algorithm~\ref{alg:bbpca} and provide our main result, Theorem~\ref{thm:final_composition}, in Section~\ref{ssec:cpca_upper}.

\subsection{Gap-free Wedin decompositions and basic cPCA composition}\label{ssec:cpca_prelims}

In this section, we develop two general facts used throughout our analysis. The first is a matrix decomposition inspired by the gap-free Wedin theorem (Lemma B.3 of \cite{allen2016lazysvd}). At a high level, Lemma~\ref{lemma:T1_T2_decomposition} applies noncommuting projection matrices compatible with only one of two matrices $\mm, \tmm$ on different sides of their difference, and rearranges terms. Its conclusion allows us to develop tools to bound the correlation between the small eigenspace of $\mm$ and the large eigenspace of $\tmm$.

\begin{lemma} \label{lemma:T1_T2_decomposition}
Let $\mm \in \PSD^{d \times d}$, let $\mmu \in \R^{d \times r}$ be orthonormal for $r \in [d]$, and let $\tmm \defeq (\id_d - \mmu\mmu^\top)\mm(\id_d - \mmu\mmu^\top)$. Suppose $\mm, \tmm$ have eigendecompositions 
\bas{
    \mm = \ml \mlam \ml^\top + \ms \msig \ms^\top, \; \tmm = \tml \tmlam \tml^\top + \tms \tmsig \tms^\top,\text{ for } \mmu \perp \tml,\; \mmu \perp \tms,
}
i.e.\ $\ml, \ms, \tml, \tms$ are all orthonormal, $\ml, \ms$ have columns forming the eigenvectors of $\mm$, and $\tml, \tms, \mmu$ have columns forming the eigenvectors of $\tmm$. Then, assuming $\tmlam \in \PD^{d \times d}$,
\bas{
    \ms^{\top}\tml &= \bb{\id_d - \ms^{\top}\mmu\mmu^\top\ms}\msig \ms^{\top} \tml\tmlam^{-1} - \ms^{\top}\mmu\mmu^\top\ml \mlam \ml^{\top}\tml\tmlam^{-1}.
}
\end{lemma}
\begin{proof}
Let $\md \defeq \tmm - \mm = (\id_d - \mmu\mmu^\top)\mm(\id_d -\mmu\mmu^\top) - \mm$.
Then,
\ba{
    \ms^{\top}\mathbf{D}\tml 
    &= \ms^{\top}\bb{\Par{\id_d - \mmu\mmu^\top}\mm\Par{\id_d - \mmu\mmu^\top}\tml - \mm\tml} \notag  \\
    &= \ms^{\top}\bb{\Par{\id_d - \mmu\mmu^\top}\mm\tml - \mm\tml} = -\ms^{\top}\mmu\mmu^\top\mm\tml \notag  \\
    &= -\ms^{\top}\mmu\mmu^\top \bb{ \ms\msig \ms^{\top} + \ml \mlam \ml^{\top}} \tml \notag  \\
    &= -\ms^{\top}\mmu\mmu^\top\ms\msig \ms^{\top} \tml - \ms^{\top}\mmu\mmu^\top\ml \mlam \ml^{\top}\tml,\label{eq:rep1}
}
where the second equality used $\tml \perp \mmu$.
We also have,
\ba{
    \ms^{\top}\mathbf{D}\tml &= \ms^{\top}\bb{\tmm\tml - \mm\tml} = \ms^{\top}\bb{\tml \tmlam - \mm\tml} = \ms^{\top}\tml \tmlam - \ms^{\top}\mm\tml = \ms^{\top}\tml \tmlam - \msig\ms^{\top}\tml \label{eq:rep2} 
}
where the second equality used $\tmm = \tml\tmlam\tml^\top + \tms\tmsig\tms^\top$ and $\tml \perp \tms$, and the fourth equality used $\mm = \ms\msig \ms^{\top} + \ml \mlam \ml^{\top}$.
Combining \eqref{eq:rep1} and \eqref{eq:rep2} and right-multiplying by $\tmlam^{-1}$, we have the claim.
\end{proof}

We next prove a useful consequence of Lemma~\ref{lemma:T1_T2_decomposition}, which is repeatedly used to analyze the compositions of approximate cPCA algorithms in Section~\ref{ssec:cpca_upper}. The proof first applies two cPCA guarantees in turn, and then uses Lemma~\ref{lemma:T1_T2_decomposition} to bound a cross term arising in the analysis.

\begin{lemma}
    \label{lemma:u1_u2_composition}
    Let $\delta_1, \delta_2, \gamma_1, \gamma_2 \in [0, \frac 1 {10}]$. Let $\mmu_{1} \in \mathbb{R}^{d \times k_{1}}$ be a $(\delta_1, \gamma_1)$-$k_1$-cPCA of $\mm \in \PSD^{d \times d}$, and define $\tmm \defeq (\id_d - \mmu_1\mmu_1^\top)\mm(\id_d - \mmu_1\mmu_1^\top)$. Let $\mmu_2 \in \R^{d \times k_2}$ be a $(\delta_2,\gamma_2)$-$k_2$-cPCA of $\tmm$ with $\mmu_2 \perp \mmu_1$, and let $\gamma \defeq \max(\gamma_1, 2\gamma_2)$, $k \defeq k_1 + k_2$. Let the eigendecompositions of $\mm$, $\tmm$ be
    \bas{
         \mm &= \ml \mlam \ml^\top + \ms \msig \ms^\top, \text{ where } \ml \defeq \mv^{\ge (1 - \gamma)\lam_{k}(\mm) }(\mm), \\
        \tmm &= \tml \tmlam \tml^\top + \tms \tmsig \tms^\top, \text{ where } \tml \defeq \mv^{\ge (1 - \gamma_2)\lam_{k_2}(\tmm)}(\tmm), \text{ and } \mmu \perp \tml,\; \mmu \perp \tms,
    }
    i.e.\ orthonormal $\ml, \ms$ have columns forming eigenvectors of $\mm$, and orthonormal $\tml, \tms, \mmu_1$ have columns forming eigenvectors of $\tmm$. Then, $\mmu \defeq \begin{pmatrix}\mmu_1 & \mmu_2 \end{pmatrix} \in \mathbb{R}^{d \times k}$ is a $(\Delta, \gamma)$-$k$-cPCA of $\mm$, with
    \bas{
       \Delta \defeq \delta_1 + 2\delta_2 + \frac{4\delta_1}{\gamma_2^2} \cdot \frac{\normop{\mmu_{1}^\top\ml \mlam \ml^{\top}\tml}^{2}}{\lambda_{k}\bb{\mm}^{2}}.
    }
\end{lemma}
\begin{proof}
We first observe
\ba{
    \normf{\ms^{\top}\mmu}^{2} = \Tr\bb{\mmu^\top\ms\ms^{\top} \mmu} = \Tr\bb{\mmu_1^\top\ms\ms^{\top} \mmu_{1}} + \Tr\bb{\mmu_2^\top\ms\ms^{\top} \mmu_{2}}.\label{eq:U1_U2_total_error}
}
Since $\mmu_{1}$ is a $\bb{\delta_{1},\gamma_{1}}$-$k_{1}$-cPCA of $\mm$, it is also a $(\delta_1, \gamma)$-$k$-cPCA (since $\gamma \geq \gamma_1$), so we have $\Tr\bb{\mmu_1^\top\ms\ms^{\top}\mmu_{1}} \leq \delta_{1}$. Next, we bound $\Tr\bb{\ms\ms^{\top} \mmu_{2}\mmu_{2}^{\top}}$. 
Then, since $\mmu_2 \perp \mmu_1$,
    \ba{
        \Tr\bb{\mmu_{2}^{\top}\ms\ms^{\top}\mmu_{2}} 
        &= \Tr\bb{\mmu_{2}^{\top}\bb{\tml\tml^{\top} + \tms\tms^{\top}}\ms\ms^{\top}\bb{\tml\tml^{\top} + \tms\tms^{\top}}\mmu_{2}} \notag \\
        &= \Tr\bb{\bb{\mmu_{2}^{\top}\tml\tml^{\top}\ms + \mmu_{2}^{\top}\tms\tms^{\top}\ms}\bb{\ms^{\top}\tml\tml^{\top}\mmu_{2} + \ms^{\top}\tms\tms^{\top}\mmu_{2}}} \notag \\
        &\leq 2\Tr\bb{\mmu_{2}^{\top}\tml\tml^{\top}\ms\ms^{\top}\tml\tml^{\top}\mmu_{2}} + 2\Tr\bb{\mmu_{2}^{\top}\tms\tms^{\top}\ms\ms^{\top}\tms\tms^{\top}\mmu_{2}}, \text{ using Fact}~\ref{fact:trace_cs} \notag \\
        &= 2\Tr\bb{ \tml^{\top}\mmu_{2}\mmu_{2}^{\top}\tml \tml^{\top}\ms\ms^{\top}\tml} + 2\Tr\bb{\tms^{\top}\ms\ms^{\top}\tms\tms^{\top}\mmu_{2}\mmu_{2}^{\top}\tms}  \notag \\
        &\leq 2\Tr\bb{ \tml^{\top}\ms\ms^{\top}\tml} + 2\Tr\bb{\tms^{\top}\mmu_{2}\mmu_{2}^{\top}\tms} , \text{using \Cref{fact:matrix-inner-prodct}, } \notag
        \\
        & \qquad \qquad \qquad \qquad\qquad \text{ $\normop{\tml^{\top}\mmu_{2}\mmu_{2}^{\top}\tml} \leq 1$, and  $\normop{\tms^\top\ms\ms^\top\tms} \leq 1$} \notag \\ 
        &\leq 2\normf{\ms^{\top}\tml}^{2} + 2\delta_{2}\label{eq:U2_S_bound} 
    }
where the last line used the cPCA guarantee.
Now we bound $\normsf{\ms^{\top}\tml}$. By Lemma \ref{lemma:T1_T2_decomposition},
\ba{
    \ms^{\top}\tml &= \bb{\id_d - \ms^{\top}\mmu_{1}\mmu_{1}^\top\ms}\msig \ms^{\top} \tml\tmlam^{-1} - \ms^{\top}\mmu_{1}\mmu_{1}^\top\ml \mlam \ml^{\top}\tml\tmlam^{-1}.\label{eq:T1_T2_decomposition}
}
We next focus on the first term.
Since $\normop{\msig} < \bb{1-\gamma}\lambda_{k}\bb{\mm}$, and $\lam_1(\ms^\top\mmu_1\mmu_1^\top \ms) \le 1$, we have
\[
\tmlam^{-1}\tml^{\top}\ms\msig\bb{\id_d - \ms^{\top}\mmu_{1}\mmu_{1}^\top\ms}^{2}\msig \ms^{\top}\tml\tmlam^{-1} \preceq (1 - \gamma)^2\lambda_{k}\bb{\mm}^{2}\tmlam^{-1}\tml^{\top}\ms\ms^{\top}\tml\tmlam^{-1}.
\]
Further, since Fact~\ref{fact:cauchy_interlacing} implies $\lam_k(\mm) \le \lam_{k_2}(\tmm)$, \Cref{fact:matrix-inner-prodct} implies that
\[\Tr\bb{\tmlam^{-1}\tml^{\top}\ms\ms^{\top}\tml\tmlam^{-1}} =\Tr\bb{\tmlam^{-2}\tml^{\top}\ms\ms^{\top}\tml} \le \normop{\tmlam^{-1}}^2 \normf{\ms^\top\tml}^2 \le \frac{1}{(1 - \gamma_2)^2\lam_k(\mm)^2}\normf{\ms^\top\tml}^2.\]
Combining the above two displays  proves $\normsf{(\id_d - \ms^\top\mmu_1\mmu_1^\top\ms)\msig\ms^\top\tml\tmlam^{-1}} \le (1 - \gamma_2)\normsf{\ms^\top\tml}$, since $\frac{(1 - \gamma)^2}{(1 - \gamma_2)^2} \le (1 - \gamma_2)^2$ for $\gamma \ge 2\gamma_2$.
Plugging this into~\eqref{eq:T1_T2_decomposition}, along with a triangle inequality, yields
\bas{
    \normf{\ms^{\top}\tml} \leq \bb{1 - \gamma_{2}}\normf{\ms^{\top}\tml} + \normf{\ms^{\top}\mmu_{1}\mmu_{1}^\top\ml \mlam \ml^{\top}\tml\tmlam^{-1}},
}
which implies
\ba{
    \normf{\ms^{\top}\tml} \leq \frac{\normf{\ms^{\top}\mmu_{1}\mmu_{1}^\top\ml \mlam \ml^{\top}\tml\tmlam^{-1}}}{\gamma_{2}}.\label{eq:SL_bound}
}
Finally, the conclusion follows by combining \eqref{eq:U1_U2_total_error}, \eqref{eq:U2_S_bound}, \eqref{eq:SL_bound} and $\normf{\ma \mb} \leq \normop{\ma}\normf{\mb}$:
\ba{
    \normf{\ms^{\top}\mmu_{1}\mmu_{1}^\top\ml \mlam \ml^{\top}\tml\tmlam^{-1}}^{2} 
    &\leq \normf{\mmu_{1}^\top\ms}^2 \normop{\mmu_{1}^\top\ml \mlam \ml^{\top}\tml\tmlam^{-1}}^2 \notag \\
    &\leq \normf{\mmu_{1}^\top\ms}^2 \normop{\tmlam^{-1}}^2 \normop{\mmu_{1}^\top\ml \mlam \ml^{\top}\tml}^2 \notag \\
    &\leq \delta_1 \frac{1}{\bb{1-\gamma_{2}}^{2}\lambda_{k_{2}}\bb{\tmm}^{2}}  \normop{\mmu_{1}^\top\ml \mlam \ml^{\top}\tml}^2 \notag \\
    &\leq \frac{\delta_{1}(1+3\gamma_{2})}{\lambda_{k}\bb{\mm}^{2}}\normop{\mmu_{1}^\top\ml \mlam \ml^{\top}\tml}^{2} \le \frac{2\delta_{1}}{\lambda_{k}\bb{\mm}^{2}}\normop{\mmu_{1}^\top\ml \mlam \ml^{\top}\tml}^{2}.
}
\end{proof}
\subsection{Invalid regimes for black-box cPCA}\label{ssec:cpca_lower}

Before we prove our main result on cPCA (Theorem~\ref{thm:final_composition}) in the following Section~\ref{ssec:cpca_upper}, in this section, we point out that not all values of $\bb{\delta,\gamma}$ pairs are feasible inputs for Algorithm~\ref{alg:bbpca} to serve as a lossless black-box reduction even for constant $k$, unlike in the case of ePCA. Formally, we characterize two types of parameter regimes for black-box cPCA using the following definition.

\begin{definition}[Black-box cPCA regimes]\label{def:invalid}
Let $g: [0, 1] \times [1,\infty) \to [0, 1]$. We say $g$ induces an \emph{invalid black-box cPCA regime} if, for every $f: \N \to \R_{> 0}$, there is a value $\Delta \in [0, 1]$, a $k \in \N$, and $\mm \in \PSD^{d \times d}$ with $d \ge k$ such that, letting $\Gamma \defeq g(\Delta,\kappa_{k}\bb{\mm})$,\footnote{We let $\kappa_k(\mm)$ denote the $k$-condition number of $\mm$, i.e.\ the ratio of $\lam_1(\mm)$ and $\lam_k(\mm)$ (see Section~\ref{sec:prelims}).} and defining
\begin{equation}\label{eq:small_dg_def}\delta \defeq \frac{\Delta}{f(k)},\; \gamma \defeq \frac{\Gamma}{f(k)},\end{equation}
$\BBPCA$ does not always return a $(\Delta, \Gamma)$-$k$-cPCA if $\oopca$ is a $(\delta,\gamma)$-$1$-cPCA oracle. 

Conversely, we say $g$ induces a \emph{valid black-box cPCA regime} if there exists $f: \N \to \R_{> 0}$, such that for all $\mm \in \PSD^{d \times d}$, $\Delta \in [0, 1]$, $\Gamma := g(\Delta, \kappa_{k}\bb{\mm})$, and $(\delta,\gam)$ in \eqref{eq:small_dg_def}, $\BBPCA$ always returns a $(\Delta,\Gam)$-$k$-cPCA if $\oopca$ is a $(\delta,\gamma)$-1-cPCA oracle.
\end{definition}
For example, suppose we establish $g$ which induces an invalid black-box cPCA regime. This means that for $\BBPCA$ to serve generically as a $(\Delta, \Gam \defeq g(\Delta,\kappa))$-$k$-cPCA algorithm for all $\Delta \in [0, 1]$ and $\kappa > 1$, we must ask for $\oopca$ to be a $(\delta, \gam)$-$1$-cPCA oracle where either $\delta = o_{k}(\Delta)$ or $\gam = o_{k}(\Gamma)$, i.e.\ we necessarily suffer some loss in either the $\Delta$ or $\Gam$ parameter by more than a function of $k$. The following impossibility result establishes functions $g$ which induce invalid black-box cPCA regimes.

Our choice to parameterize the regimes $g$ in Definition~\ref{def:invalid} by $\kappa_k(\mm)$ is motivated by a result of \cite{liang2023optimality} (see Eq.\ (1.3) in that paper), which proves that some sample complexity overhead in $\kappa_k(\mm)$ is necessary in general. On the other hand, the requirement that $\Gamma \gtrsim \sqrt{\Delta}$ appears to be an artifact of the specific deflation strategy in Algorithm~\ref{alg:bbpca}. We now state our main impossibility result.

\begin{restatable}[Invalid black-box cPCA regimes]{proposition}{restateinvalid}\label{prop:delta_gg_gamsquare_bad}
Let $\alpha, \beta, \nu \in \R_{> 0}$, and $g: [0, 1] \times (1, \infty) \to [0, 1]$ be defined as $g\bb{\Delta, \kappa} \defeq \nu \Delta^{\alpha}\kappa^{\beta}$. If $\alpha > \frac{1}{2}$ or $\beta < 1$, $g$ induces an invalid black-box cPCA regime.
\end{restatable}

Moreover, Theorem~\ref{thm:final_composition} shows that the complement of the regimes ruled out by Proposition~\ref{prop:delta_gg_gamsquare_bad} are all valid regimes, where we suffer parameter degradation only in $k$. Hence, Proposition~\ref{prop:delta_gg_gamsquare_bad} and Theorem~\ref{thm:final_composition} completely characterize valid and invalid black-box cPCA regimes, under Definition~\ref{def:invalid}.

We establish Proposition~\ref{prop:delta_gg_gamsquare_bad} by giving $3$-dimensional examples where Algorithm~\ref{alg:bbpca} fails to return a $(\Delta,\Gamma)$-$2$-cPCA if $\Gamma \lesssim \kappa_2(\mm) \sqrt{\Delta}$ asymptotically, assuming $\oopca$ is a $(\delta,\gamma)$-$1$-cPCA oracle, for any $\delta,\gamma$ smaller than $\Delta,\Gamma$ by a constant. First, we show that sublinear functions induce invalid regimes.

\begin{lemma}\label{lem:delta_gg_gam_bad}
Let $\alpha, \beta , \nu \in \R_{> 0}$, and define $g: [0, 1] \times (1, \infty) \to [0, 1]$ as $g\bb{\Delta, \kappa} \defeq \nu \Delta^{\alpha}\kappa^{\beta}$. If $\alpha > 1$ or $\beta < 1$, then $g$ induces an invalid black-box cPCA regime.
\end{lemma}

\begin{proof}
Following the notation in \Cref{def:invalid}, fix $f: \N \to \R_{> 0}$. Without loss of generality, we assume that $f$ is nondecreasing.
Let $C := f(2)$. We divide the analysis into two cases:
\begin{enumerate}
    \item $\bm{\alpha > 1:}$ For any $\kappa > 1$, $\frac{g(\Delta, \kappa)}{\Delta \bb{\kappa-1}} \le \min\bb{\frac{1}{2C}, \frac{1}{4}}$ for all sufficiently small $\Delta$. 
    \item $\bm{\beta < 1:}$ For any $\Delta > 0$, $\frac{g(\Delta, \kappa)}{\Delta \bb{\kappa-1}} \le \min\bb{\frac{1}{2C}, \frac{1}{4}}$ for all sufficiently large $\kappa$.
\end{enumerate}
Therefore, we can always find a pair $\bb{\Delta, \kappa}$ in both cases, such that $\frac{g(\Delta, \kappa)}{\Delta \bb{\kappa-1}} \le \min\bb{\frac{1}{2C}, \frac{1}{4}}$. We will show an example with $k=2$ and a matrix $\mm$ with $k$-condition number $\kappa_{k}\bb{\mm} = \kappa$. If we can show that for $\delta = \frac{\Delta}{C}$, two $(\delta, 0)$-1-cPCA oracles fail to produce a $(\Delta, \Gamma)$-1-cPCA, where $\Gamma \defeq g(\Delta, \kappa)$, the same will be true for calling two $(\frac{\Delta}{f(k)}, \frac{\Gamma}{f(k)})$-1-cPCA oracles, since $\delta = \frac{\Delta}{f(k)}$ and $0 \le \frac{\Gamma}{f(k)}$. This gives us our desired contradiction, since it shows $g$ induces an invalid black-box cPCA regime.

To this end, we now show that there is $\mm \in \PSD^{3 \times 3}$ with $\kappa_{k}(\mm) = \kappa$ such that calling two $(\delta, 0)$-$1$-cPCA oracles as in Algorithm~\ref{alg:bbpca} fails to return a $(\Delta, \Gamma)$-$2$-cPCA. Define the $3\times 3$ diagonal matrix
\[\mm \defeq \begin{pmatrix}\kappa & 0 & 0 \\ 0 & 1 & 0 \\ 0 & 0 & 1 - 2\Gamma\end{pmatrix}.\]
We let $u_1^\top = \begin{pmatrix} \sqrt{1 - \delta} & 0 & \sqrt{\delta} \end{pmatrix}$ and observe $u_1$ is a $(\delta,0)$-1-cPCA of $\mm$. We also let $u_2$ be the top eigenvector of $\tmm \defeq (\id_3 - u_1u_1^\top)\mm(\id_3 - u_1u_1^\top)$, which is a $(\delta,0)$-1-cPCA of $\tmm$. To compute $u_2$, note
\[\tmm = \begin{pmatrix} \delta((1 - 2\Gam)(1 - \delta) + \kappa\delta) & 0 & -\sqrt{\delta(1 - \delta)}((1 - 2\Gam)(1 - \delta) + \kappa\delta) \\ 0 & 1 & 0 \\ -\sqrt{\delta(1 - \delta)}((1 - 2\Gam)(1 - \delta) + \kappa\delta) & 0 & (1 - \delta)((1 - 2\Gam)(1 - \delta) + \kappa\delta) \end{pmatrix},\]
so that the eigenvector-eigenvalue pairs of $\tmm$ are
\[\Par{e_2, 1},\; \Par{\begin{pmatrix}-\sqrt{\delta} \\ 0 \\ \sqrt{1 - \delta} \end{pmatrix}, (1 - 2\Gam)(1 - \delta) + \kappa\delta}.\]
Since $\delta = \frac {\Delta}{C} \ge \frac{2\Gam}{\kappa-1}$, we have 
\bas{
(1 - 2\Gam)(1 - \delta) + \kappa\delta &= 1 + 2\Gamma\delta + \bb{\kappa-1}\delta - 2\Gamma \geq 1,
} so $u_2^\top = \begin{pmatrix} -\sqrt{\delta} & 0 & \sqrt{1 - \delta}\end{pmatrix}$ is an exact $1$-cPCA of $\tmm$. In this case, $\mmu = \begin{pmatrix} u_1 & u_2\end{pmatrix}$ fails to be a $(\Delta, \Gam)$-$2$-cPCA of $\mm$ as claimed, since
\[\normf{\mmu^\top \mv^{< (1 - \Gamma)\lam_2(\mm)}}^2 = \normf{\mmu^\top e_3}^2 = 1 \ge \Delta.\]
The second equality can also be seen by noting that $e_3 \in \linspan(\mmu)$.
\end{proof}

We next improve Lemma~\ref{lem:delta_gg_gam_bad} to show that we can rule out $g$ asymptotically smaller than $\sqrt{\cdot}$ in $\Delta$. Namely for a fixed $\kappa$, while Lemma~\ref{lem:delta_gg_gam_bad} rules out black-box cPCA reductions without parameter losses for $\Gamma \ll \Delta$, when $\Delta$ is small enough, Proposition~\ref{prop:delta_gg_gamsquare_bad} further rules out such reductions for $\Gamma \ll \sqrt{\Delta}$.

\restateinvalid*
\begin{proof}
If $\alpha > 1$ or $\beta < 1$, then applying Lemma~\ref{lem:delta_gg_gam_bad} yields the claim. Thus, in the remainder of the proof suppose $\frac{1}{2} < \alpha \le 1$ and $\beta \ge 1$. Fix $f: \N \to \R_{> 0}$, and define $C \defeq f(2)$.

Since $\alpha \le 1$, there exists a $c > 0$ such that $g(\Delta, 2) \ge c\Delta$ for all sufficiently small $\Delta$.
Let 
\[K \defeq \min\bb{\frac{c}{10}, \frac{1}{10C}, \frac{1}{100}}.\]
Since $\alpha > \frac{1}{2}$, $g(\Delta, 2) \le K\sqrt{\Delta}$ for all sufficiently small $\Delta$.

Assume $\Delta < 1$ is small enough that $\Gamma \defeq g(\Delta)$ satisfies
\[\Gamma \in \Brack{c\Delta , K \sqrt{\Delta}}. \]
Let $\delta := 10K\Delta \le \frac{\Delta}{C}$. Note that $\delta \le c\Delta \le \Gamma$. Let
\bas{
    \mm \defeq \begin{pmatrix}
                2 & 0 & 0 \\
                0 & 1 & 0 \\
                0 & 0 & 1-2\Gamma
            \end{pmatrix},\; \tmm \defeq \Par{\id_3 - u_1u_1^\top}\mm\Par{\id_3 - u_1u_1^\top}.
}

We prove $g$ induces an invalid black-box cPCA regime by showing that calling two $\bb{\frac{\Delta}{C}, 0}$-1-cPCA oracles on $\mm$ as in Algorithm~\ref{alg:bbpca} fails to return a $(\Delta, \Gamma)$-2-cPCA. 

Define $u_1 = \begin{pmatrix} \sqrt{1 - \delta} & \sqrt{\delta/2} & \sqrt{\delta/2}\end{pmatrix}^\top$, which is a $(\frac{\Delta}{C}, 0)$-$1$-cPCA for $\mm$. Let $u_2 = \begin{pmatrix} u_{21} & u_{22} & u_{23} \end{pmatrix}^\top$ be the top eigenvector of $\tmm$ and $\lambda$ be the corresponding eigenvalue, so $u_2$ is a $(\frac{\Delta}{C},0)$-$1$-cPCA for $\tmm$. Since $\mv^{< (1 - \Gam)\lam_2(\mm)}(\mm)= e_3$, to prove $\mmu = \begin{pmatrix} u_1 & u_2 \end{pmatrix}$ is not a $(\Delta, \Gam)$-cPCA of $\mm$ we must show
\begin{equation}\label{eq:u23big}u_{23}^2 > \Delta.\end{equation}
By a direct calculation,\footnote{We provide a calculation of the eigenvalue \href{https://colab.research.google.com/drive/1_FYClOLE2PlBMYaiIDlCtxt7GjA4K1lq?usp=sharing}{here}.} we get
\ba{
    \lambda &= 1 + \frac{\delta}{2} + \frac{\delta\Gamma}{2} + \bb{\frac{\sqrt{4\Gamma^{2} + \delta^{2} + \delta\Gamma\bb{2\delta + \delta\Gamma- 4\Gamma}}}{2} - \Gamma}. \label{eq:lambda_tmm_equal}
}
Next, we provide upper and lower bounds on $\lambda$, which will be useful for the rest of the analysis.
\bas{
\sqrt{4\Gamma^{2} + \delta^{2} + \delta\Gamma\bb{2\delta + \delta\Gamma- 4\Gamma}} \le \sqrt{4\Gamma^{2} + \delta^{2} + 4\delta\Gamma} = 2\Gamma+\delta, 
}
and
\bas{
\sqrt{4\Gamma^{2} + \delta^{2} + \delta\Gamma\bb{2\delta + \delta\Gamma- 4\Gamma}} &\ge \sqrt{4\Gamma^2 + \delta^2 -4\Gamma^2\delta} \ge 2\Gamma,
}
where the last inequality used $\delta = 10K\Delta \ge 4K^2\Delta \ge 4\Gamma^2$.
Using these bounds in \eqref{eq:lambda_tmm_equal}, we have
\ba{
1+\frac{\delta}{2} < \lambda < 1+\frac{3\delta}{2}.\label{eq:lambda_tmm_bound}
}
Since $u_1$ and $u_2$ are distinct eigenvectors of $\tmm \succeq 0$, they are orthogonal to each other. Therefore,
\bas{
    \lambda u_2 = \tmm u_2 = \bb{\id_3 - u_1u_1^{\top}}\mm\bb{\id_3 - u_1u_1^{\top}}u_2  = \mm u_2 - \bb{u_1^{\top}\mm u_2}u_1. \notag
}
Since $\mm$ is diagonal, rearranging the above equality and writing $u_1$, $u_2$ in the canonical basis,
\bas{
\begin{pmatrix}
            \bb{2 - \lambda}u_{21}  \\
            \bb{1 - \lambda}u_{22} \\
            \bb{1 - 2\Gamma - \lambda}u_{23}
        \end{pmatrix} = \bb{u_1^{\top}\mm u_2}\begin{pmatrix}
            \sqrt{1-\delta} \\
            \sqrt{\frac{\delta}{2}} \\
            \sqrt{\frac{\delta}{2}}
        \end{pmatrix}.
}
Using the bound on $\lambda$ from \eqref{eq:lambda_tmm_bound}, the above equation implies that $u_1^{\top}\mm u_2, u_{21}, u_{22},$ and $u_{23}$ are all nonzero, and $\lambda \notin \{2,1,1-2\Gamma\}$. Therefore,
\ba{
    \frac{u_{21}}{u_{22}}  = \sqrt{\frac{2\bb{1-\delta}}{\delta}} \bb{\frac{1-\lambda}{2 - \lambda}}, ~\frac{u_{23}}{u_{22}} 
 = \bb{\frac{\lambda-1}{\lambda-1+2\Gamma}}. \label{eq:u2_comp_relation}
}
Taking absolute values and using the bounds on $\lambda$ from \eqref{eq:lambda_tmm_bound},
\ba{
    \left | \frac{u_{21}}{u_{22}} \right | \le \sqrt{\frac{2\bb{1-\delta}}{\delta}} \bb{\frac{\bb{1 + \frac{3\delta}{2}}-1}{2 - \bb{1 + \frac{3\delta}{2}}}} = \sqrt{\frac{2\bb{1-\delta}}{\delta}} \bb{\frac{3\delta}{2-3\delta}} \le 1, \label{eq:w1_bound}
}
and
\ba{
    \left | \frac{u_{23}}{u_{22}} \right |  = \frac{\lambda-1}{\lambda-1+2\Gamma} \le 1,\label{eq:w3_bound}
}
where \eqref{eq:w1_bound} used that $\delta = 10K\Delta \leq \frac 1 {10}$. 
Inequalities \eqref{eq:w1_bound} and \eqref{eq:w3_bound} imply $|u_{22}|=\max\{|u_{21}|, |u_{22}|, |u_{23}|\}$. Since $u_2$ is a unit vector, $u_{22}^2 \ge \frac{1}{3}$. Using \eqref{eq:lambda_tmm_bound} and \eqref{eq:u2_comp_relation}, our claim \eqref{eq:u23big} follows as shown below:
\bas{
u_{23}^2 &= \bb{\frac{\lambda-1}{\lambda-1+2\Gamma}}^2 u_{22}^2 \ge \bb{\frac{\bb{1+\frac{\delta}{2}}-1}{\bb{1+\frac{\delta}{2}}-1+2\Gamma}}^2 \cdot \frac{1}{3} \\
&= \frac{\delta^2}{3\bb{\delta+4\Gamma}^2} \ge \frac{\delta^2}{75\Gamma^2} \ge \frac{100K^2\Delta^2}{75\bb{K\sqrt{\Delta}}^2} > \Delta.
}
\end{proof}

\subsection{Black-box cPCA in the valid regime}\label{ssec:cpca_upper}

In this section, we provide composition results used to bound the performance of Algorithm~\ref{alg:bbpca} in the regime $\Gamma \ge \sqrt{\Delta} \cdot \kappa_{k}\bb{\mm}$, which is the asymptotically slowest decay rate on $\Gamma$ not ruled out by Proposition~\ref{prop:delta_gg_gamsquare_bad}. Our analysis of the composition of black-box cPCAs proceeds in two steps.
\begin{enumerate}
    \item In Lemma~\ref{lem:gap_comp}, we first consider the case where we perform a sequence of recursive cPCAs on $\mm$, and there exist gaps in the spectrum of $\mm$. We show how in this gapped case, cPCAs are composable with only a $\poly(k)$ blowup in cPCA parameters by applying Lemma~\ref{lemma:u1_u2_composition}.
    \item In Lemma~\ref{lem:nogap_merge_two}, we then consider the gap-free case, where we use a careful argument on head guarantees (see Definition~\ref{def:head_guarantee}) that we show are recursively afforded by approximate cPCAs, to bound the parameter blowup by a quasipolynomial factor in $k$.
\end{enumerate}

\subsubsection{cPCA composition: the gapped case}

In this section, we provide a composition lemma which applies under gap assumptions in the spectrum.
Before providing our full gapped composition result, we state two helper lemmas.

\begin{lemma}\label{lem:gap_merge_two}
In the setting of Lemma~\ref{lemma:u1_u2_composition}, suppose that $(1 - \gamma_1)\lam_{k_1}(\mm) > \lam_{k_1 + 1}(\mm)$ and $\delta_1 \le \frac{\gamma_2^2}{16\kappa_k(\mm)^2}$. Then, $\mmu$ is a $(2(\delta_1 + \delta_2), \gamma)$-$k$-cPCA of $\mm$, for $\gamma = \max(\gamma_1, 2\gamma_2)$.
\end{lemma}
\begin{proof}
By Lemma~\ref{lemma:u1_u2_composition}, it suffices to provide a bound on $\normsop{\mmu_1^\top \ml \mlam \ml^\top \tml}$. 

Throughout the proof, let $w \in \linspan(\tml) \subseteq \linspan([\mmu_{1}]_{\perp})$, $u \in \linspan(\mmu_1)$ be arbitrary unit vectors, and note that $u \perp w$ by definition. Let $\rank\bb{\ml}  = k' \ge k$. Then,
\begin{equation}\label{eq:r1r2_def}
\begin{aligned}
u^\top \ml \mlam \ml^\top w &= \sum_{i \in [k']} \lam_i(\mm) [\ml^\top u]_i[\ml^\top w]_i \\
&= \underbrace{\sum_{i =1}^{k_1} \lam_i(\mm) [\ml^\top u]_i [\ml^\top w]_i}_{\defeq R_1} + \underbrace{\sum_{i = k_1 + 1}^{k'} \lam_i(\mm) [\ml^\top u]_i [\ml^\top w]_i}_{\defeq R_2}.
\end{aligned}
\end{equation}
Let orthonormal $\ml_1 \in \R^{d \times k_1}$ span the subspace of eigenvectors corresponding to the largest $k_1$ eigenvalues of $\mm$, and let $\ms_1$ be the complement subspace. 
We assumed that $(1 - \gamma_1)\lam_{k_1}(\mm) > \lam_{k_1 + 1}(\mm)$, so that $\normsf{\ms_1^\top\mmu_1}^2 \le \delta_1$ by the cPCA guarantee. This also lets us conclude
\begin{equation}\label{eq:gapped_head}
\begin{aligned}
\normf{\ml_1^\top \mmu_1}^2 &= \normf{\mmu_1}^2 - \normf{\ms_1^\top \mmu_1}^2 \ge k_1 - \delta_1 \\
\implies \norm{\ml_1^\top w}_2^2 &\le \normf{\ml_1}^2 - \normf{\ml_1^\top \mmu_1}^2 \le \delta_1 \text{ for all } w \in \R^d,\; w \perp \mmu_1.
\end{aligned}
\end{equation}
Let $\mlam_1 \in \R^{k_1 \times k_1}$ and $\mlam_2 \in \R^{k_2 \times k_2}$ be diagonal matrices, so that $\mlam_1$ has diagonal entries $\{\lam_i(\mm)\}_{i \in [k_1]}$ and $\mlam_2$ has diagonal entries $\{\lam_i(\mm)\}_{i \in [k] \setminus [k_1]}$. We now bound the terms in \eqref{eq:r1r2_def}: first,
\begin{align*}
|R_1| = \Abs{\Par{\ml_1^\top u}^\top \mlam_1 \Par{\ml_1^\top w}} \le \normop{\mlam_1}\norm{\ml_1^\top w}_2\norm{\ml_1^\top u}_2 \le \lam_1(\mm)\sqrt{\delta_1}.
\end{align*}
In the last inequality we used \eqref{eq:gapped_head}. Next, we similarly have
\begin{align*}
|R_2|
 \le \normop{\mlam_2}\norm{\ms_1^\top w}_2\norm{\ms_1^\top u}_2 \le \lam_1(\mm)\sqrt{\delta_1},
\end{align*}
since $\norms{\ms_1^\top u}_2 \le \normsf{\ms_1^\top \mmu_1} \le \sqrt{\delta_1}$. Plugging in the above two displays into \eqref{eq:r1r2_def}, and applying Lemma~\ref{lemma:u1_u2_composition}, then yields the result due to the assumed bound on $\frac{\delta_1}{\gamma_2}$:
\[\normf{\ms^\top \mmu}^2 \le \delta_1 + 2\delta_2 + \frac{4\delta_1}{\gamma_2^2\lam_k(\mm)^2} \cdot \Par{2\sqrt{\delta_1}\lam_1(\mm)}^2 = \delta_1 + 2\delta_2 + \frac{16\delta_1^2 \kappa_k(\mm)^2}{\gamma_2^2} \le 2(\delta_1+\delta_2).\]
\end{proof}

\begin{lemma}\label{lem:opnorm_close}
Let $\mm \in \PSD^{d \times d}$ have $\lam_{k + 1}(\mm) < (1 - \Gam)\lam_k(\mm)$ for $k \in [d]$ and $\Gam \in (0, 1)$, let $\mv \in \R^{d \times k}$ be an exact $k$-PCA of $\mm$, and let $\mmu \in \R^{d \times k}$ be a $(\Delta, \Gam)$-$k$-cPCA of $\mm$. Then,
\[\normop{\Par{\id_d - \mmu\mmu^\top}\mm\Par{\id_d - \mmu\mmu^\top} - \Par{\id_d - \mv\mv^\top}\mm\Par{\id_d - \mv\mv^\top}} \le 4\sqrt{\Delta}\lam_1(\mm).\]
\end{lemma}
\begin{proof}
By our gap assumption, the cPCA guarantees imply $\normsf{(\id_d - \mv\mv^\top)\mmu\mmu^\top}^2 \le \Delta$, so that
\[\normf{\Par{\id_d -\mmu\mmu^\top}\mv\mv^\top}^2 = k - \inprod{\mmu\mmu^\top}{\mv\mv^\top} = \normf{\Par{\id_d - \mv\mv^\top}\mmu\mmu^\top}^2 \le \Delta.\]
Next, observe that, since $\mmu\mmu^\top - \mv\mv^\top = \mmu\mmu^\top(\id_d - \mv\mv^\top) - (\id_d - \mmu\mmu^\top)\mv\mv^\top$,
\begin{align*}
\normop{\mmu\mmu^\top - \mv\mv^\top} &\le \normop{\mmu\mmu^\top(\id_d - \mv\mv^\top)} + \normop{(\id_d - \mmu\mmu^\top)\mv\mv^\top} \le 2\sqrt{\Delta},
\end{align*}
where we used our earlier bounds. The conclusion follows from
\begin{gather*}
\Par{\id_d - \mmu\mmu^\top}\mm\Par{\id_d - \mmu\mmu^\top} - \Par{\id_d - \mv\mv^\top}\mm\Par{\id_d - \mv\mv^\top} \\
= \Par{\id_d - \mv\mv^\top}\mm\Par{\mv\mv^\top - \mmu\mmu^\top} - \Par{\mmu\mmu^\top - \mv\mv^\top}\mm\Par{\id_d - \mmu\mmu^\top},
\end{gather*}
after using the triangle inequality and our earlier bound on $\normsop{\mmu\mmu^\top - \mv\mv^\top}$.
\end{proof}

We now apply Lemma~\ref{lem:gap_merge_two} recursively to obtain our gapped composition result, aided by Lemma~\ref{lem:opnorm_close}.

\begin{lemma}\label{lem:gap_comp}
Let $k \in [d]$ and $r \in [k]$, such that $\sum_{j \in [r]} k_j = k$ for $\{k_j\}_{j \in [r]} \subset \N$. Suppose for all $j \in [r]$, the following recursively hold for some $\delta, \gamma \in [0, \frac 1 {10}]$, where we initialize $\mproj_0 \gets \id_{d}$.
\begin{enumerate}
    \item $\mmu_j$ is a $(\delta, \gamma)$-approximate $k_j$-cPCA of $\mproj_{j - 1} \mm \mproj_{j - 1}$ with $\linspan(\mmu_j) \subseteq \linspan(\mproj_{j - 1})$. 
    \item $\mproj_j = \mproj_{j - 1} - \mmu_j\mmu_j^\top$.
\end{enumerate}
Suppose for $\Delta \defeq 4r^2 \delta$, $\Gamma \defeq 2r\gamma$, $\max(\Delta, \Gamma) \le \frac 1 {10}$, and letting $K_j \defeq \sum_{j' \in [j]} k_{j'}$ for all $j \in [r]$,
\begin{equation}\label{eq:gap_assume} \lam_{K_j + 1}\Par{\mm} < \Par{1 - \Gamma}\lam_{K_j}\Par{\mm} \text{ for all } j \in [r - 1],\text{ and }  \Delta \le \frac{\Gamma^2}{64\kappa_k(\mm)^2}.\end{equation}
Then $\mmu = \begin{pmatrix} \mmu_1 & \ldots & \mmu_r \end{pmatrix}$ is a $(\Delta, \Gamma)$-approximate $k$-cPCA of $\mm$.
\end{lemma}
\begin{proof}
For all $j \in [r]$, let $\mw_j \defeq \begin{pmatrix} \mmu_1 & \ldots & \mmu_j\end{pmatrix} \in \R^{d \times K_j}$, so that $\mmu = \mw_k$. Our proof will be by double induction. Throughout the proof, refer to the following assumption by $\Prefix_{r - 1}$: for all $j \in [r - 1]$, $\mw_j$ is a $(\Delta, \Gamma)$-$K_j$-cPCA of $\mm$. We will prove the stated claim assuming $\Prefix_{r - 1}$, which means that $\Prefix_{r - 1}$ implies $\Prefix_r$. Because all the assumptions continue to apply if the lemma statement took smaller values of $r$ and $k$, this means we can induct on the base of $\Prefix$, justifying our assumption of $\Prefix_{r - 1}$. For the remainder of the proof, suppose $\Prefix_{r - 1}$ is true.

We first state a consequence of $\Prefix_{r - 1}$. Let $1 \le j' \le j \le r - 1$, so $\mw_{j'}$ is a $(\Delta, \Gamma)$-$K_{j'}$-cPCA of $\mm$ by assumption, and $\mw_{j'}\mw_{j'}^\top = \id_d - \mproj_{j'}$. Moreover, let $\mv_{j'} \in \R^{d \times K_{j'}}$ be an exact $K_{j'}$-cPCA of $\mm$, which is unique by the gap assumption \eqref{eq:gap_assume}, and let $\mq_{j'} \defeq \mv_{j'}\mv_{j'}^\top$ for convenience. We have:
\begin{equation}\label{eq:gap_still_holds}
\begin{aligned}
\lam_{K_j - K_{j'} + 1}\Par{\Par{\id_d - \mproj_{j'}}\mm\Par{\id_d - \mproj_{j'}}
} &\le \lam_{K_j - K_{j'} + 1}\Par{\Par{\id_d - \mq_{j'}}\mm\Par{\id_d - \mq_{j'}}} \\
&+ \normop{\Par{\id_d - \mproj_{j'}}\mm\Par{\id_d - \mproj_{j'}} - \Par{\id_d - \mq_{j'}}\mm\Par{\id_d - \mq_{j'}}} \\
&\le \lam_{K_j + 1}(\mm) + 4\sqrt{\Delta}\lam_1(\mm) \\
&< \Par{1 - \Gamma}\lam_{K_j}(\mm) + \frac{\Gamma}{2}\lam_{k}(\mm) \le \Par{1 - \frac{\Gam}{2}}\lam_{K_j}(\mm) \\
&\le \Par{1 - \frac{\Gamma}{2}}\lam_{K_j - K_{j'}}\Par{\Par{\id_d - \mproj_{j'}}\mm\Par{\id_d - \mproj_{j'}}
}.
\end{aligned}    
\end{equation}
Here, the first inequality was by Fact~\ref{fact:weyl}, the second used Lemma~\ref{lem:opnorm_close} and the gap assumption \eqref{eq:gap_assume} with $j \gets j'$, the third used \eqref{eq:gap_assume} with $j \gets j$, the fourth used $\lam_k(\mm) \le \lam_{K_j}(\mm)$, and the last used Fact~\ref{fact:cauchy_interlacing}. In other words, \eqref{eq:gap_still_holds} shows that after any $j'$ steps of our recursive cPCA procedure, all remaining gaps (ensured for our original matrix by \eqref{eq:gap_assume}) continue to hold in the residual matrix $(\id_d - \mproj_{j'})\mm(\id_d - \mproj_{j'})$, up to a small multiplicative loss in the gap parameter.

Next, we divide $[r]$ into dyadic intervals, using $L + 1$ layers for $L \defeq \lceil\log_2(r)\rceil$, labeled $0 \le \ell \le L$. In particular, the $0^\text{th}$ layer is $S_{0, 1} \defeq [r]$, the first layer consists of the intervals $S_{1, 0} \defeq [2^{L - 1}]$ and $S_{1, 1} \defeq [r] \setminus S_{1, 0}$, and so on. More generally, for each $0\le \ell \le L$ and $0 \le i < 2^{\ell}$, we let
\begin{gather*}S_{\ell, i} \defeq \Brace{j \in [r] \mid i 2^{L - \ell} + 1 \le j \le (i + 1)2^{L - \ell}}, \\
a_{\ell, i} \defeq \sum_{j \in [i2^{L - \ell}]} k_j = K_{i2^{L - \ell}},\; b_{\ell, i} \defeq \sum_{j \in [\min(r, (i + 1)2^{L - \ell})]} k_j = K_{\min(r, (i + 1)2^{L - \ell})}.
\end{gather*}
In other words, this dyadic splitting induces a binary tree, where the ($L^{\text{th}}$-layer) leaves correspond to single elements in $[r]$, and the ($0^{\text{th}}$-layer) root corresponds to the set $[r]$. Moreover, the node in the tree associated with $S_{\ell, i} \subseteq [r]$ captures the interval $[a_{\ell, i} + 1, b_{\ell, i}] \subseteq [k]$. We refer to this node as the $(\ell, i)^{\text{th}}$ node, and we associate it with $\mw_{\ell, i}$, an approximate $(b_{\ell, i} - a_{\ell, i})$-cPCA of the matrix $\mproj_{i2^{L - \ell}} \mm \mproj_{i2^{L - \ell}}$. In particular, $\mw_{\ell, i}$ is just the horizontal concatenation of $\mmu_j$ for all $j \in S_{\ell, i}$.
We inductively analyze the approximation quality of the intermediate $\mw_{\ell, i}$, to show that for all $0 \le \ell \le L$, $\mw_{\ell, i}$ is a $(\delta_\ell, \gamma_\ell)$-approximate $(b_{\ell, i} - a_{\ell, i})$-cPCA of $\mproj_{i2^{L - \ell}} \mm \mproj_{i2^{L - \ell}}$ for all $0 \le i < 2^\ell$ and
\[\delta_\ell \defeq 4^{L - \ell}\delta,\; \gamma_\ell \defeq 2^{L - \ell}\gamma.\]

The base case, $\ell = L$, follows since all $\mw_{L, i} = \mmu_{i + 1}$ are $(\delta, \gamma)$-approximate cPCAs of their corresponding $\mproj_{i} \mm \mproj_{i}$. Next, suppose inductively that all the $\mw_{\ell, i}$ are $(\delta_\ell, \gamma_\ell)$-approximate cPCAs for some $0 \le \ell < L$, and all $0 \le i < 2^\ell$. Then, consider some $\mw_{\ell - 1, i}$, which is the composition of $\mw_{\ell, 2i}$ and $\mw_{\ell, 2i + 1}$. To analyze this composition, we apply Lemma~\ref{lem:gap_merge_two}, which requires that
\begin{equation}\label{eq:residual_gap}\lam_{b_{\ell, 2i} - a_{\ell, 2i}}\Par{\mproj_{i2^{L - \ell + 1}}\mm\mproj_{i2^{L - \ell + 1}}} \ge (1 + \gamma_\ell)\lam_{b_{\ell - 1, i} - a_{\ell - 1, i}}\Par{\mproj_{i2^{L - \ell + 1}}\mm\mproj_{i2^{L - \ell + 1}}}.\end{equation}
Because $\gamma_\ell \le \gamma_0 \le \frac \Gamma 2$ for all $0 \le \ell < L$, \eqref{eq:residual_gap} is implied by \eqref{eq:gap_still_holds}. Lemma~\ref{lem:gap_merge_two} also requires $\delta_\ell \le \frac{\gamma_\ell^2}{16\kappa_k(\mm)^2}$, which is invariant to the choice of $\ell$ since $\delta_\ell$ and $\gamma_\ell^2$ grow at the same rate, so this is implied by our bounds in \eqref{eq:gap_assume}. Therefore, we can apply Lemma~\ref{lem:gap_merge_two} and $\mw_{\ell - 1, i}$ is indeed a $(4\delta_\ell, 2\gamma_\ell) = (\delta_{\ell - 1}, \gamma_{\ell - 1})$-cPCA as claimed. The conclusion follows by taking $\ell = 0$, as $4^L \le 4r^2$, $2^L \le 2r$.
\end{proof}

The following byproduct of our proof of Lemma~\ref{lem:gap_comp} is useful in our later development.

\begin{lemma}\label{lem:well_conditioned_pieces}
In the setting of Lemma~\ref{lem:gap_comp}, suppose that $\lam_{K_j}(\mm) \ge \frac 2 3 \lam_{K_{j - 1} + 1}(\mm)$ for all $j \in [r]$, where we let $K_0 \defeq 0$. Then, for all $j \in [r]$, we have $\kappa_{k_j}\Par{\mproj_{j - 1}\mm\mproj_{j - 1}} \le 2$.
\end{lemma}
\begin{proof}
First of all, we have $\lam_{k_j}(\mproj_{j - 1}\mm\mproj_{j - 1}) \ge \lam_{K_j}(\mm)$ by Fact~\ref{fact:cauchy_interlacing}. Moreover, an analogous argument to \eqref{eq:gap_still_holds} shows that the largest eigenvalue of $\mproj_{j - 1}\mm\mproj_{j - 1}$ is perturbed by at most $\frac \Gam 2 \lam_k(\mm) \le \frac 1 {20} \lam_{K_{j - 1} + 1}(\mm)$, when compared to $\lam_{K_{j - 1} + 1}(\mm)$, so $\lam_{1}(\mproj_{j - 1}\mm\mproj_{j - 1}) \le 1.05\lam_{K_{j - 1} + 1}(\mm)$. The conclusion follows from combining these inequalities, because $1.05 \cdot \frac 3 2 \le 2$.
\end{proof}

\subsubsection{cPCA composition: the well-conditioned case}

In this section, we analyze Algorithm~\ref{alg:bbpca} under the promise that $\kappa_k(\mm) \le 2$. We begin with a basic helper lemma on composition under a head guarantee (Definition~\ref{def:head_guarantee}), patterned off of Lemma~\ref{lem:gap_merge_two}.

\begin{lemma}\label{lem:nogap_merge_one}
In the setting of Lemma~\ref{lemma:u1_u2_composition}, let $\delta_1 \le \frac{ \gamma_2^2}{288}$, $\kappa_k(\mm) \le 2$, and $\gamma_1 \le \gamma_2$. For $0 \le h \le d$, suppose that $\mmu_1$ satisfies a $(h, 2k_1 \gamma_1, \delta_1)$-head guarantee (Definition~\ref{def:head_guarantee}) with respect to $\mm$.
Then, $\mmu$ is a $(\delta, \gamma)$-$k$-cPCA of $\mm$, for $\delta := 130k_1^2\delta_{1} + 2\delta_{2}$ and $\gamma := 2\gamma_2$.
\end{lemma}
\begin{proof}
Throughout the proof, for convenience we denote
\[\ml_h \defeq \mv^{\ge \lam_h(\mm)}\Par{\mm} \implies \normf{\ml_{h}^{\top}\Brack{\mmu_{1}}_{\perp}}^{2} \leq \delta_{1}\text{ and }  \frac{\lambda_{h + 1}\bb{\mm}}{\lambda_{k_1}\bb{\mm}} \le 1 + 2k_1\gamma_1.\]
By Lemma~\ref{lemma:u1_u2_composition}, it suffices to bound $\normsop{\mmu_1^\top \ml \mlam \ml^\top \tml}$. Throughout the proof, let $w \in \linspan(\tml) \subseteq \linspan([\mmu_{1}]_{\perp})$, $u \in \linspan(\mmu_1)$ be arbitrary unit vectors, so $u \perp w$.
Next, let $m \in [d]$ be the largest index such that $\lambda_{m}\bb{\mm} \geq \bb{1-\gamma_1}\lambda_{k_{1}}\bb{\mm}$. Finally, let $k' \defeq \rank\bb{\ml} \geq k$. Note that $h \le k_1 \le m \le k'$, where $h \le k_1$ follows because otherwise $\normsop{\ml_h^\top[\mmu_1]_\perp} < 1$ is impossible, as this would mean
\[\dim\Par{\linspan\Par{\ml_h}} + \dim\Par{\linspan\Par{\Brack{\mmu_1}_\perp}} = h + d - k_1 > d \implies \linspan(\ml_h) \cap \linspan([\mmu_1]_\perp) \neq \emptyset. \]
Next, letting $\mv \defeq \begin{pmatrix}\ml & \ms \end{pmatrix}$ be a full set of orthonormal eigenvectors for $\mm$, we have
\begin{equation}\label{eq:ULLamLTML_bucket_bound_1}
\begin{aligned}
     u^\top\ml \mlam \ml^{\top}w
    &= \sum_{i \in \bbb{k'}}\lambda_{i}\bb{\mm} [\ml^\top u]_i [\ml^\top w]_i \\
    &= \sum_{i \in \bbb{m}} \lambda_{i}\bb{\mm} [\ml^\top u]_i [\ml^\top w]_i + \sum_{i \in \bbb{k'} \setminus \bbb{m}} \lambda_{i}\bb{\mm} [\ml^\top u]_i [\ml^\top w]_i - \lambda_{k_{1}}\bb{\mm}\inprod{\mv^\top u}{\mv^\top w} \\
    &=  \underbrace{\sum_{i \in \bbb{h}}\bb{\lambda_{i}\bb{\mm}-\lambda_{k_{1}}\bb{\mm}} [\ml^\top u]_i [\ml^\top w]_i}_{:= R_{1}} \\
    &+ \underbrace{\sum_{i \in \bbb{m}\setminus\bbb{h}}\bb{\lambda_{i}\bb{\mm}-\lambda_{k_{1}}\bb{\mm}} [\ml^\top u]_i [\ml^\top w]_i}_{:= R_{2}}  \\
    & + \underbrace{\sum_{i \in \bbb{k'} \setminus \bbb{m}} \lambda_{i}\bb{\mm} [\ml^\top u]_i [\ml^\top w]_i}_{:= R_{3}} - \underbrace{\sum_{i \in [d] \setminus [m]}\lambda_{k_{1}}\bb{\mm} [\mv^\top u]_i [\mv^\top w]_i}_{:= R_{4}} .
\end{aligned}
\end{equation}
In the second equality, we used $\inprod{u}{w} = 0$ and $\mv\mv^\top = \id_d$, and in the third, we used that the first $m$ rows of $\mv^\top$ agree with $\ml^\top$. For convenience in the following, let $\mlam_1 \in \R^{\bb{h-1} \times \bb{h-1}}$, $\mlam_2 \in \R^{\bb{m-h+1} \times \bb{m-h+1}}$, $\mlam_3 \in \R^{\bb{k'-m} \times \bb{k'-m}}$ and $\mlam_{4} \in \R^{\bb{d-m} \times \bb{d-m}}$ be diagonal matrices such that $\mlam_1$ has diagonal entries $\left\{\lambda_{i}\bb{\mm}-\lambda_{k_{1}}\bb{\mm}\right\}_{i \in \bbb{h-1}}$, $\mlam_2$ has diagonal entries $\left\{\lambda_{i}\bb{\mm}-\lambda_{k_{1}}\bb{\mm}\right\}_{i \in \bbb{m} \setminus \bbb{h-1}}$, $\mlam_3$ has diagonal entries $\left\{\lambda_{i}\bb{\mm}\right\}_{i \in \bbb{k'} \setminus \bbb{m}}$ and $\mlam_4$ has diagonal entries $\left\{\lambda_{k_{1}}\bb{\mm}\right\}_{i \in \bbb{d} \setminus \bbb{m}}$.

We now bound the terms in \eqref{eq:ULLamLTML_bucket_bound_1}. For $R_{1}$,
\begin{align*}
|R_1| = \Abs{\Par{\ml_h^\top u}^\top \mlam_1 \Par{\ml_h^\top w}} \le \normop{\mlam_1}\norm{\ml_h^\top w}_2\norm{\ml_h^\top u}_2 \le \lam_1(\mm)\sqrt{\delta_1}.
\end{align*}
The last inequality used $\norms{\ml_h^\top w}_2 \leq \normf{\ml_{h}^{\top}[\mmu_{1}]_{\perp}} \leq \sqrt{\delta_{1}}$.
For $R_{2}$, since $\norm{u}_2 = \norm{w}_2 = 1$,
\begin{align*}
|R_2| \le \normop{\mlam_2} \le \max\left\{\frac{\lambda_{h + 1}\bb{\mm}}{\lambda_{k_{1}}\bb{\mm}} - 1, \gamma_{1}\right\}\lam_{k_{1}}(\mm).
\end{align*}
For $R_{3}, R_{4}$, since $\norm{\ms^{\top}u}_{2} \leq \normf{\ms^{\top}\mmu_{1}} \leq \sqrt{\delta_{1}}$ by the cPCA guarantee on $\mmu_1$, and the $j^{\text{th}}$ rows of $\mv^\top$ and $\ms^\top$ agree for any $j > m$ by definition,
\bas{
    |R_3| &\le \normop{\mlam_3}\norm{\ms^\top w}_2\norm{\ms^\top u}_2 \le \lam_{1}(\mm)\sqrt{\delta_1},  \\
    |R_4| &\le \normop{\mlam_4}\norm{\ms^\top w}_2\norm{\ms^\top u}_2 \le \lam_{1}(\mm)\sqrt{\delta_1}.
}
Plugging in the above displays into \eqref{eq:ULLamLTML_bucket_bound_1}, we have
\bas{
    \normsop{\mmu_1^\top \ml \mlam \ml^\top \tml}^{2} &\leq \bb{3\lam_1(\mm)\sqrt{\delta_1} + \max\left\{\frac{\lambda_{h+1}\bb{\mm}}{\lambda_{k_{1}}\bb{\mm}} - 1, \gamma_{1}\right\}\lam_{k_{1}}(\mm)}^{2} \\
    &\leq 18\lambda_{1}\bb{\mm}^{2}\delta_1 + 2\max\left\{\bb{2k_1\gamma_1}^{2}, \gamma_{1}^{2}\right\}\lam_{k_{1}}(\mm)^{2} \\
    &= 18\lambda_{1}\bb{\mm}^{2}\delta_1 + 8k_1^2\gamma_1^2\lam_{1}(\mm)^{2}.
}
Applying Lemma~\ref{lemma:u1_u2_composition} and using the assumed bounds in the lemma statement then yields the result:
\bas{
\normf{\ms^\top \mmu}^2 &\le \delta_1 + 2\delta_2 + \frac{4\delta_1}{\gamma_2^2\lam_k(\mm)^2} \cdot \Par{18\lambda_{1}\bb{\mm}^{2}\delta_1 + 8k_1^2\gamma_1^2 \lam_{1}(\mm)^{2}} \\
&= \delta_1 + 2\delta_2 + \frac{8\delta_1 \kappa_{k}\bb{\mm}^{2}}{\gamma_2^2} \cdot \Par{9\delta_1 + 4k_1^2\gamma_1^2} \leq 130k_1^2 \delta_{1} + 2\delta_{2}.
}
\end{proof}

In the absence of an explicit gap in the spectrum of $\mm$, Lemma~\ref{lem:nogap_merge_one} shows how to nonetheless apply Lemma~\ref{lemma:u1_u2_composition}, assuming a head guarantee. The following lemma shows how to inductively use cPCA bounds to enforce such a head guarantee. Informally, the requirement \eqref{eq:prefix_plus} states that after $i$ steps of Algorithm~\ref{alg:bbpca}, the next $h$ steps are a cPCA of the residual matrix, for all $h \in [m]$. We show that \eqref{eq:prefix_plus} either implies a gap in the spectrum of the residual matrix, or that the entire residual matrix is well-conditioned in the top-$m$ subspace. Either case yields a guarantee compatible with Lemma~\ref{lem:nogap_merge_one}.

\begin{lemma}
    \label{lem:no_gap_head_guarantee}
    Let $\mm \in \PSD^{d \times d}$ and let $\mmu \in \R^{d \times d}$ be orthonormal, with columns $\{u_\ell\}_{\ell \in [d]}$. For $1 \le i \le j \le d$, let $\mmu_{[i,j]} \in \R^{d \times (j - i + 1)}$ have columns $\{u_\ell\}_{\ell \in [i, j]}$. Further, for $i \in [d]$, suppose for all $h \in [m]$,    \begin{equation}\label{eq:prefix_plus}
    \begin{gathered}
        \mmu_{[i+1, i+h]} \text{ is a } \bb{\delta,\gamma}\text{-}h\text{-cPCA of } \tmm_i \defeq \bb{\id_d-\mmu_{[1, i]}\mmu_{[1,i]}^\top} \mm \bb{\id_d -\mmu_{[1, i]}\mmu_{[1,i]}^\top}. 
    \end{gathered}
    \end{equation}
    There is $h \in [0, m - 1]$ so $\mmu_{[i + 1, i + m]}$ satisfies a $(h, 2m\gamma, \delta)$-head guarantee with respect to $\tmm_i$.
\end{lemma}
\begin{proof}
Let $h \in [1, m - 1]$ be maximal such that $\lam_{h + 1}(\tmm_i) \le (1 - \gamma)\lam_{h}(\tmm_i)$; if no such index exists, then we set $h = 0$. Then, by definition the first condition in \eqref{eq:head_conditions} is satisfied for this value $h$ if $h \neq 0$, and otherwise it is irrelevant. Further, since $h$ is the largest such index, $\lam_{j + 1}(\tmm_i) > (1 - \gamma)\lam_{j}(\tmm_i)$ for all $j \in [h + 1, m - 1]$, so telescoping across all such $j$, we have the third claim in \eqref{eq:head_conditions}:
\[\frac{\lam_{h + 1}(\tmm_i)}{\lam_m(\tmm_i)} \le \Par{\frac 1 {1- \gamma}}^{m - h - 1} \le 1 + 2m\gamma.\]
This holds even if $h = 0$, and the second claim in \eqref{eq:head_conditions} is irrelevant in this case, so we assume $h \ge 1$ henceforth. For notational convenience, let $\tml_{i, h}  \defeq \mv^{\ge \lam_h(\tmm_i)}(\tmm_i) \in \R^{d \times h}$. 
The definition of $h$ implies $
\tml_{i, h} = [\mv^{\le (1 - \gamma)\lam_{h}(\tmm_i)}(\tmm_i)]_\perp$, so \eqref{eq:prefix_plus} with $m \gets h$ yields
\begin{align*}
\normf{\mmu_{[i + 1, i + h]}^\top \mv^{\le (1 - \gamma)\lam_{h}(\tmm_i)}(\tmm_i)}^2 \le \delta &\implies \normf{\mmu_{[i + 1, i + h]}^\top\tml_{i, h}}^2 \ge h - \delta \\
&\implies \normf{\tml_{i, h}^\top \Brack{\mmu_{[i + 1, i + h]}}_{\perp}}^2 \leq \delta \\
&\implies \normf{\tml_{i, h}^\top \Brack{\mmu_{[i + 1, i + m]}}_{\perp}}^2 \leq \delta.
\end{align*}
In the last claim, we used that $\linspan(\mmu_{[i + 1, i + h]}) \supset \linspan(\mmu_{[i + 1, i + m]})$.
\end{proof}

Finally, we apply Lemmas~\ref{lem:nogap_merge_one} and~\ref{lem:no_gap_head_guarantee} to complete our analysis of Algorithm~\ref{alg:bbpca}.

\begin{lemma} \label{lem:nogap_merge_two}
Let $k \in [d]$ and $\kappa_k(\mm) \le 2$. Let $\mmu \in \R^{d \times k}$ be the output of $\BBPCA\bb{\mm, k, \orcpca}$ (Algorithm \ref{alg:bbpca}) where $\orcpca$ is a $(\delta, \gamma)$-$1$-cPCA oracle. Let
\[\Delta \defeq \Par{132k^2}^{\lceil \log_2(k) \rceil}\delta,\; \Gamma \defeq 2^{\lceil \log_2(k) \rceil}\gamma.\]
Assume $\Delta \le \frac{\Gamma^2}{288}$.
Then $\mmu$ is a $(\Delta, \Gamma)$-$k$-cPCA of $\mm$.
\end{lemma}
\begin{proof}

We divide $[k]$ into dyadic intervals as in Lemma \ref{lem:gap_comp}, using $L + 1$ layers for $L \defeq \lceil\log_2(k)\rceil$, labeled $0 \le \ell \le L$. In particular, the $0^\text{th}$ layer is $S_{0, 1} \defeq [k]$, the first layer consists of the intervals $S_{1, 0} \defeq [2^{L - 1}]$ and $S_{1, 1} \defeq [k] \setminus S_{1, 0}$, and so on. For each $0\le \ell \le L$ and $0 \le i < 2^{\ell}$, we let
\begin{gather*}S_{\ell, i} \defeq \Brace{j \in [k] \mid i 2^{L - \ell} + 1 \le j \le (i + 1)2^{L - \ell}}.
\end{gather*}
This dyadic splitting again induces a binary tree, where singletons in $[k]$ are leaves and the root is $[k]$. The node in the tree associated with $S_{\ell, i} \subseteq [k]$ captures the interval $[i2^{L - \ell} + 1, \min((i + 1)2^{L - \ell}, k)] \subseteq [k]$, and (following notation in Lemma~\ref{lem:no_gap_head_guarantee}) we associate it with 
\[\mw_{i, \ell} \defeq \mmu_{[i2^{L - \ell} + 1, \min((i + 1)2^{L - \ell}, k)]}.\] 
Moreover, we define $m_{i, \ell} \defeq \dim(\linspan(\mw_{i, \ell}))$, and for all $h \in [m_{i, \ell}]$, we let
\[\mw_{i, \ell}^{(h)} \defeq \mmu_{[i2^{L - \ell} + 1, i2^{L - \ell} + h]},\]
i.e.\ $\mw_{i, \ell}^{(h)}$ is the first $h$ columns of $\mw_{i, \ell}$. We inductively claim that for all $0 \le \ell \le L$, $\mw_{\ell, i}^{(h)}$ is a $(\delta_\ell, \gamma_\ell)$-approximate $h$-cPCA of $\mproj_{i2^{L - \ell}}\mm\mproj_{i2^{L - \ell}}$, for all $0 \le i < 2^\ell$, and all $h \in [m_{i, \ell}]$, where
\[\delta_\ell \defeq \Par{132k^2}^{L - \ell}\delta,\; \gamma_\ell \defeq 2^{L - \ell}\gamma.\]
The base case $\ell = L$ follows by assumption. Next, suppose the claim above holds for some $0 \le \ell < L$, and all $0 \le i < 2^\ell$. Consider some $\mw_{\ell - 1, i}$, the composition of $\mw_{\ell, 2i}$ and $\mw_{\ell, 2i + 1}$. If $\mw_{\ell, 2i + 1}$ is empty, then all $\mw_{\ell - 1, i}^{(h)}$ are $(\delta_{\ell - 1}, \gamma_{\ell - 1})$-$h$-cPCAs of $\mproj_{i2^{L - \ell + 1}}\mm\mproj_{i2^{L - \ell + 1}}$ by the inductive assumption. 

Otherwise, by induction, \eqref{eq:prefix_plus} holds with $i \gets i2^{L - \ell + 1}$ and $m \gets 2^{L - \ell}$, for $\delta \gets \delta_\ell$, $\gamma \gets \gamma_\ell$. So, we can apply Lemma~\ref{lem:no_gap_head_guarantee} to obtain $h \in [0, 2^{L - \ell}-1]$ satisfying the head guarantee needed by Lemma~\ref{lem:nogap_merge_one}. We then use Lemma~\ref{lem:nogap_merge_one} with $\mm \gets \mproj_{i2^{L - \ell + 1}}\mm\mproj_{i2^{L - \ell + 1}}$ to analyze our composition. In particular, $\kappa_{k'}(\mproj_{i2^{L - \ell + 1}}\mm\mproj_{i2^{L - \ell + 1}}) \le 2$ where $k' \defeq \min(2^{L - \ell + 1}, k - i2^{L - \ell + 1})$, by Fact~\ref{fact:cauchy_interlacing} and $\kappa_k(\mm) \le 2$. Moreover, $\delta_\ell$ grows faster than $\gamma_\ell^2$, so the remaining requirement in Lemma~\ref{lem:nogap_merge_one} is satisfied by our choices of $\Delta, \Gamma$. Finally, Lemma~\ref{lem:nogap_merge_one} proves that $\mw_{i, \ell - 1}$ is a $(\delta_{\ell - 1}, \gamma_{\ell - 1})$-$m_{i, \ell-1}$-cPCA of $\mproj_{i2^{L - \ell + 1}}\mm\mproj_{i2^{L - \ell + 1}}$. It is straightforward to check the same argument we used also shows that $\mw_{i, \ell - 1}^{(h)}$ is a $(\delta_{\ell - 1}, \gamma_{\ell - 1})$-$h$-cPCA of $\mproj_{i2^{L - \ell + 1}}\mm\mproj_{i2^{L - \ell + 1}}$ for all $h \in [m_{i, \ell - 1}]$, since we can truncate the composition early. This completes the induction, and the conclusion follows by taking $\ell \gets 0$.
\end{proof}

\subsubsection{Putting it all together}
\label{subsubsection:putting_it_all_together}
In this section, we finally combine the results from the previous two sections, and provide a generic analysis on the parameter degradation of $\BBPCA$'s guarantees when using $1$-cPCA oracles. 

\restatecpcathm*
\begin{proof}
For convenience in this proof, let
\begin{align*}
\bDelta \defeq \frac{\Delta}{640 k^2},\; \Delta' \defeq \frac{\bDelta}{18k^2},\; \delta \defeq \frac{\Delta'}{\Par{132k^2}^{\lceil \log_2(k)\rceil}} , \;
\bGamma \defeq \frac{\Gamma}{10k},\; \Gamma' \defeq \frac{\bGamma}{2k},\; \gamma \defeq \frac {\Gamma'} {2^{\lceil\log_2(k)\rceil}}.
\end{align*}
Also, let $\gapind$ be the set of indices $i \in [\min(k, d - 1)]$ such that $\lam_{i + 1}(\mm) < (1 - \bGamma)\lam_i(\mm)$, and let the elements of $\gapind$ be denoted $\{K_j\}_{j \in [r - 1]}$. Let $K_0 \defeq 0$ and $K_r \defeq k$, and define $k_j \defeq K_j - K_{j - 1}$ for all $j \in [r]$. Following the notation of $\oopca$, we let $\mmu_j \defeq \{u_i\}_{i \in [K_{j - 1} + 1, K_j]} \in \R^{d \times k_i}$, and we iteratively define $\mproj^{(0)} \defeq \id_d$, $\mproj^{(j)} \defeq \mproj^{(j - 1)} - \mmu_j \mmu_j^\top$ for all $j \in [r - 1]$. We claim inductively that for all $j \in [r]$, $\mmu_j$ is a $(\Delta', \Gamma')$-approximate $k_j$-cPCA of $\mproj^{(j - 1)}\mm \mproj^{(j - 1)}$. Because $\bGamma \le \frac 1 {10k}$, we have
\[\frac{\lam_{K_j}(\mm)}{\lam_{K_{j - 1} + 1}(\mm)} = \prod_{i = K_{j - 1} + 1}^{K_j - 1} \frac{\lam_{i + 1}(\mm)}{\lam_i(\mm)} > \Par{1 - \bGamma}^k \ge \frac 2 3.\]
It is also straightforward to check that $\bDelta, \bGamma$ satisfy the requirements of Lemma~\ref{lem:gap_comp}. Therefore, the inductive assumption and Lemma~\ref{lem:well_conditioned_pieces} show that $\kappa_{k_j}(\mproj^{(j - 1)}\mm\mproj^{(j - 1)}) \le 2$. By applying Lemma~\ref{lem:nogap_merge_two} and $\Delta' \le\frac 1 {288} \cdot (\Gamma')^2$, we hence have completed our inductive argument. Finally, applying Lemma~\ref{lem:gap_comp} shows that $\mmu$ is a $(\bDelta, \bGamma)$-$k$-cPCA, so it is clearly also a $(\Delta, \Gamma)$-cPCA as claimed.
\end{proof}
\section{Applications}\label{sec:applications}

In this section, we give applications of our main results, Theorems~\ref{thm:epca_reduce} and~\ref{thm:final_composition}, to the design of $k$-PCA algorithms for statistical problems. In Section~\ref{ssec:app_prelims}, we begin by proving several preliminary results which will be used throughout our applications. We provide our results on robust $k$-PCA under various distributional assumptions, Theorems~\ref{thm:robust_subg} and~\ref{thm:robust-k-epca-heavy-tailed}, in Section~\ref{ssec:robust}. Finally, in Section~\ref{ssec:htoja}, we develop a simple $k$-cPCA algorithm in a heavy-tailed online setting in Theorem~\ref{thm:oja_k}.

\subsection{Preliminaries}\label{ssec:app_prelims}

In this section, we establish some results about the clipping of sampled data points that will be used to handle heavy tails in subsequent applications. We first define the standard notion of  $\bb{p,C_{p}}$-hypercontractivity (see, for example,  \cite{Mendelson2018RobustCE} where it was used in a similar context).

\begin{definition}[$\bb{p,C_{p}}$-hypercontractivity]
\label{def:hypercontractivity}
Let $p \ge 4$ be a positive even integer. A distribution $\mathcal{D}$ over $\R^d$ is said to be \emph{$\bb{p,C_{p}}$-hypercontractive} if for all $v \in \mathbb{R}^d$,
\bas{
\E_{X \sim \mathcal{D}} \bbb{\inprod{X-\E\bbb{X}}{v}^{p}}^{\frac{1}{p}} \le C_{p}\E_{X \sim \mathcal{D}}\bbb{\inprod{X-\E\bbb{X}}{v}^{2}}^{\half}.
}
\end{definition}
We observe that Definition~\ref{def:hypercontractivity} is compatible with assuming $\dist$ is centered and symmetric, with no loss in parameters. Namely, let $\{X_i\}_{i \in [n]} \sim \mathcal{D}$ with covariance $\E\Bracks{\bb{X-\E{X}}\bb{X-\E{X}}^{\top}} = \mathbf{\Sigma}$. If $\dist$ is not symmetric or centered, we can consider the $\frac n 2$ independent variables 
\[Y_{i} := \frac{X_{2i-1}-X_{2i}}{\sqrt{2}},\; i \in \Brack{1,\left\lfloor \frac{n}{2} \right\rfloor}.\] 
Indeed, if $X$ and $X'$ are independent with probability distribution $\mathcal{D}$ and $Y := (X-X')/\sqrt{2}$, then $\E Y = \0_d, \E\bbb{YY^\top} = \msig$, $Y$ is symmetric about $\0_d$, and
\bas{
\E_{X,X' \sim \mathcal{D}}\bbb{\inner{Y}{v}^p}^{\frac{1}{p}} &= \frac{1}{\sqrt{2}} \E_{X,X' \sim \mathcal{D}}\bbb{\bb{\inner{X-\E X}{v}-\inner{X'-\E X'}{v}}^p}^{\frac{1}{p}} \\
&\le \frac{1}{\sqrt{2}} \E_{X,X' \sim \mathcal{D}}\bbb{2^{p-1}\bb{\inner{X-\E X}{v}^p + \inner{X'-\E X'}{v}^p}}^{\frac{1}{p}} \\
&= \sqrt{2} \E_{X \sim \mathcal{D}}\bbb{{\inner{X-\E X}{v}^p}}^{\frac{1}{p}} \le C_p\E_{X \sim \mathcal{D}}\bbb{2{\inner{X-\E X}{v}^2}}^{\frac{1}{2}} \\
&= C_p\E_{X,X' \sim \mathcal{D}}\bbb{{\inner{X-X'}{v}^2}}^{\frac{1}{2}},
}
where the inequality follows from Fact \ref{fact:generalized_mean}. So, $Y$ also has a $(p,C_p)$-hypercontractive distribution. Therefore, in applications for hypercontractive $\dist$, we assume without loss that $\dist$ is centered and symmetric. We next define a clipping operation relevant to our heavy-tailed applications.

\begin{definition}[Clipping]\label{def:clipping}
Let $R \in \R_{> 0}$. The $R$-clipping function $\mathcal{T}_R: \R^d \to \R^d$ is defined as 
\bas{
\mathcal{T}_R(x) := \min\bb{1, \frac{\sqrt{R}}{\norm{x}_{2}}}v
}
for all $x \in \R^d$.
Given distribution $\mathcal{D}$ over $\R^d$, the $R$-clipped distribution of $\mathcal{D}$ is defined as $\mathcal{T}_R(\mathcal{D})$.
\end{definition}
We will often use $\mathcal{T}$ for the $R$-clipping function if the clipping radius $R$ is obvious from context. In the next result, we analyze the bias due to the clipping function.
\begin{lemma}[Bias of clipping]\label{lemma:clippingclose}
Let $\mathcal{D}$ be $(p,C_{p})$-hypercontractive over $\mathbb{R}^{d}$ with covariance $\mathbf{\Sigma}$. 
\begin{enumerate}
    \item For all $u \in \R^d$, $u^{\top}\Pars{\mathbf{\Sigma} - \E_{x \sim \mathcal{D}}\Bracks{\mathcal{T}_{R}\bb{x}\mathcal{T}_{R}\bb{x}^{\top}}}u  \leq C_{p}^{p}\bb{u^{\top}\mathbf{\Sigma}u}\Pars{\frac{\Tr\bb{\mathbf{\Sigma}}}{R}}^{\frac{p}{2}-1} $.
    \item For $x \sim \mathcal{D}$, $\Pr\Pars{\norm{x}_{2} \geq \sqrt{R}} \leq \Pars{C_{p}(\frac{\Tr\bb{\mathbf{\Sigma}}}{R})^{\half}}^{\bb{p-2}}$.
\end{enumerate}
\end{lemma}

\begin{proof}
    For $x \sim \mathcal{D}$ and any unit vector $u \in \mathbb{R}^{d}$ we have, 
    \ba{
        \E\bbb{\inprod{u}{x}^{2} - \inprod{u}{\mathcal{T}_{R}\bb{x}}^{2}} &= \E\bbb{\inprod{u}{x}^{2}\bb{1-\frac{R}{\norm{x}_{2}^{2} }}\mathbbm{1}\bb{\norm{x}_{2} \geq \sqrt{R}}} \notag \\
        &\leq \E\bbb{\inprod{u}{x}^{p}}^{\frac 2 p} \E\bbb{\bb{\mathbbm{1}\bb{\norm{x}_{2} \geq \sqrt{R}}}^{\frac p {p - 2}}}^{\frac{p-2}{p}}, \text{using Holder's inequality} \notag \\
        &\leq C_{p}^{2}\E\bbb{\inprod{u}{x}^{2}} \E\bbb{\mathbbm{1}\bb{\norm{x}_{2} \geq \sqrt{R}}}^{\frac{p-2}{p}}, \text{using } (p,C_{p})\text{-Hypercontractivity} \notag \\
        &\leq C_{p}^{2}\E\bbb{\inprod{u}{x}^{2}}\Pr\bb{\norm{x}_{2} \geq \sqrt{R}}^{\frac{p-2}{p}} \notag \\
        &\leq C_{p}^{2}\E\bbb{\inprod{u}{x}^{2}}\bb{\frac{\E\bbb{\norm{x}_{2}^{p}}}{R^{\frac{p}{2}}}}^{\frac{p-2}{p}}, \text{ using Markov's inequality} \label{eq:operator_norm_deviation}
    }
    Denote $r:= \frac{p}{2}$ for convenience of notation. Then, 
    \bas{
        \E\bbb{\norm{x}_{2}^{p}} &= \E\bbb{\bb{x^{\top}x}^{r}} = \E\bbb{\bb{\sum_{i \in [d]}x_{i}^{2}}^{r}} \\
        &= \sum_{j_{1}, j_{2}, \ldots j_{r} \in [d]}\E\bbb{x_{j_{1}}^{2}x_{j_{2}}^{2}\ldots x_{j_{r}}^{2}} \\
        &\leq \sum_{j_{1}, j_{2}, \ldots j_{r} \in [d]}\E\bbb{x_{j_{1}}^{2r}}^{\frac{1}{r}}\E\bbb{x_{j_{2}}^{2r}}^{\frac{1}{r}}\ldots \E\bbb{x_{j_{r}}^{2r}}^{\frac{1}{r}}, \text{ using Holder's inequality} \\
        &= \sum_{j_{1}, j_{2}, \ldots j_{r} \in [d]}\E\bbb{\inprod{x}{e_{j_{1}}}^{2r}}^{\frac{1}{r}}\E\bbb{\inprod{x}{e_{j_{2}}}^{2r}}^{\frac{1}{r}}\ldots \E\bbb{\inprod{x}{e_{j_{r}}}^{2r}}^{\frac{1}{r}} \\
        &\leq C_{p}^{2r}\sum_{j_{1}, j_{2}, \ldots j_{r} \in [d]}\E\bbb{\inprod{x}{e_{j_{1}}}^{2}}\E\bbb{\inprod{x}{e_{j_{2}}}^{2}}\ldots \E\bbb{\inprod{x}{e_{j_{r}}}^{2}}, \text{using } (p,C_{p})\text{-hypercontractivity} \\
        &= C_{p}^{2r}\sum_{j_{1}, j_{2}, \ldots j_{r} \in [d]}\bb{e_{j_{1}}^{\top}\mathbf{\Sigma} e_{j_{1}}}\bb{e_{j_{2}}^{\top}\mathbf{\Sigma} e_{j_{2}}}\ldots \bb{e_{j_{r}}^{\top}\mathbf{\Sigma} e_{j_{r}}} \\
        &= C_{p}^{2r}\bb{\sum_{i\in[d]}e_{i}^{\top}\mathbf{\Sigma} e_{i}}^{r} = C_{p}^{p}\Tr\bb{\mathbf{\Sigma}}^{\frac{p}{2}}.
    }
    Substituting this into~\eqref{eq:operator_norm_deviation} we have proven the second claim, and
    \bas{
        \E\bbb{\inprod{u}{x}^{2} - \inprod{u}{\mathcal{T}_{R}\bb{x}}^{2}} &\leq C_{p}^{2}\E\bbb{\inprod{u}{x}^{2}}\bb{\frac{C_{p}^{p}\Tr\bb{\mathbf{\Sigma}}^{\frac{p}{2}}}{R^{\frac{p}{2}}}}^{\frac{p-2}{p}} = C_{p}^{p}\E\bbb{\inprod{u}{x}^{2}}\bb{\frac{\Tr\bb{\mathbf{\Sigma}}}{R}}^{\frac{p}{2}-1}
    }
    which completes our proof of the first claim as well.
\end{proof}
As an immediate consequence of Lemma~\ref{lemma:clippingclose} we have the following corollary.
\begin{corollary}
\label{cor:clipping-and-truncation}
In the setting of Lemma \ref{lemma:clippingclose}, for any $\rho \in (0,\half)$ and $R\geq \Pars{\frac{C_{p}^{p}}{\rho}}^{\frac{2}{p-2}}\Tr\bb{\mathbf{\Sigma}}$, 
we have 
\begin{enumerate}
    \item $\normsop{\E_{x \sim \mathcal{D}}\Bracks{\mathcal{T}_{R}\bb{x}\mathcal{T}_{R}\bb{x}^{\top}} - \mathbf{\Sigma}} \leq \rho\normop{\msig}$. \label{item:opnormbound}
    \item $\Pr_{x \sim \mathcal{D}}\Pars{\norm{x}_{2} \geq \sqrt{R}} \leq \Pars{\frac{\rho}{C_{p}^{2}}}^{\frac{p}{p-2}}$. \label{item:bigxbound}
    \item For all $v \in \mathbb{R}^{d}$, $\E_{x \sim \mathcal{D}} \Bracks{\inprod{\mathcal{T}_{R}\bb{x}}{v}^{p}}^{\frac{1}{p}} \le 2C_{p}\E_{x \sim \mathcal{D}}\Bracks{\inprod{\mathcal{T}_{R}\bb{x}}{v}^{2}}^{\half}$.\label{item:still_hyperc}
\end{enumerate}
\begin{proof}
    Items~\ref{item:opnormbound} and~\ref{item:bigxbound} follow directly from substituting the value of $R$ in Lemma~\ref{lemma:clippingclose}. We next prove Item~\ref{item:still_hyperc}. 
     For $x \sim \mathcal{D}$ we upper bound the $p^{\text{th}}$ moment of one-dimensional projections:
     \ba{
         \E\bbb{\inprod{\mathcal{T}_{R}\bb{x}}{v}^{p}} &= \E\bbb{\inprod{x}{v}^{p}\bb{ \mathbbm{1}\bb{\norm{X}_{2} \leq \sqrt{R}} + \bb{\frac{\sqrt{R}}{\norm{x}_{2}}}^{p}\mathbbm{1}\bb{\norm{X}_{2} \geq \sqrt{R}} }} \notag \\
         &\leq 2\E\bbb{\inprod{x}{v}^{p}} \notag \\
         &\leq 2C_{p}^{p}\E\bbb{\inprod{x}{v}^{2}}^{\frac{p}{2}}, \text{ using } (p,C_{p})\text{-hypercontractivity of } X \label{eq:truncated_hypercontractivity_1}
     }
    We now obtain lower bounds on the second moments of one-dimensional projections:
     \ba{
        \E\bbb{\inprod{x}{v}^{2} - \inprod{\mathcal{T}_{R}\bb{x}}{v}^{2}} &\leq C_{p}^{p}\bb{u^{\top}\mathbf{\Sigma}u}\bb{\frac{\Tr\bb{\mathbf{\Sigma}}}{R}}^{\frac{p}{2}-1} \text{ using Lemma } \ref{lemma:clippingclose},\notag \\ 
        &\leq C_{p}^{2}\E\bbb{\inprod{X}{v}^{2}}\bb{\bb{\frac{\rho}{C_{p}^{2}}}^{\frac{p}{p-2}}}^{\frac{p - 2}{p}}= \rho \E\bbb{\inprod{X}{v}^{2}}.\notag
    }
    Substituting this into~\eqref{eq:truncated_hypercontractivity_1} completes our proof:
    \bas{
        \E\bbb{\inprod{\mathcal{T}_{R}\bb{x}}{v}^{p}} \leq 2C_{p}^{p}\bb{\frac{\E\bbb{\inprod{x}{v}^{2}}}{1-\rho}}^{\frac{p}{2}} \leq \bb{2C_{p}}^{p}\E\bbb{\inprod{x}{v}^{2}}^{\frac{p}{2}}\,.
    }
\end{proof}

\end{corollary}

\subsection{Robust PCA}\label{ssec:robust}

In this section, we develop algorithms for robust PCA. The robust PCA problem asks us to output an approximate PCA of a covariance matrix from independent draws from the inducing distribution, even after a fraction of draws are corrupted.
We define our contamination model formally below.

\begin{definition}[Strong contamination model]
\label{def:contaminaton-model}
Given a \emph{corruption parameter} $\eps \in (0,\half)$ 
and a distribution $\calP$, 
an algorithm obtains samples from $\calP$ with \emph{$\eps$-contamination} 
as follows.
\begin{enumerate}
    \item The algorithm specifies the number $n$ of samples it requires. 
    \item A set $S$ of $n$ i.i.d.\ samples from $\calP$ is drawn but not yet shown to the algorithm. 
    \item An arbitrarily powerful adversary then inspects $S$, 
before deciding to replace any subset of $\lceil \eps n \rceil$ 
samples with arbitrarily corrupted points (``outliers'') to obtain the contaminated set $T$, which is then returned to the algorithm.
\end{enumerate}
We say  $T$ is an \emph{$\epsilon$-corrupted version} of $S$ and a set of
\emph{$\epsilon$-corrupted samples} from $\calP$.
\end{definition}
Robust estimation typically requires that the inliers satisfy structural properties (that hold with high probability) that distinguishes them from harmful outliers. 
In the context of PCA, we use the following stability condition from \cite{JamLT20, DKPP23} to design our algorithms.
\begin{definition}[Stability]\label{def:stability}
For $\eps \in (0, \half)$ and $\gamma \geq \eps$, we say a distribution $G$ on $\R^d$
is \emph{$(\eps,\gamma)$-stable with respect to $\msig \in \PSD^{d \times d}$} if for all functions $w: \R^d \to [0,1]$ such that $\E_{X \sim G}[w(X)] \geq 1 -\eps$, the weighted second moment matrix $\mathbf \Sigma_{G_w} := \frac{\E_{X \sim G}[w(X)XX^\top]}{\E_{X \sim G}[w(X)]}$ satisfies $(1 - \gamma) \mathbf \Sigma \preceq \mathbf \Sigma_{G_w} \preceq (1 + \gamma) \mathbf \Sigma$.
\end{definition}
In \cite{JamLT20, DKPP23}, the distribution $G$ in Definition~\ref{def:stability} corresponds to the uniform distribution over the remaining inliers, i.e.\ $S\cap T$ in \Cref{def:contaminaton-model}.
In this notation, the uniform distribution on the corrupted set $T$ can be written as $(1-\eps )G + \eps B$, where $B$ corresponds to the uniform distribution over the outliers $T\setminus S$.
Under the stability condition stated in Definition~\ref{def:stability}, the following result from \cite{DKPP23} gives a generic 1-ePCA algorithm that runs in nearly-linear time.
\begin{proposition}[{\cite[Theorem 3.1]{DKPP23}}]
    \label{prop:robust-ePCA-stability}
    Let $\eps_0 , \gamma_0$ be sufficiently small absolute constants. 
    Let $T$ be a set of $n$ data points in $\R^d$.
    For $\eps \in (0,\eps_0)$ and $\gamma \in (0,\gamma_0)$,
    suppose the uniform distribution on $T$ can be written as $(1 - \eps) G + \eps B$, where $G$ is $(\eps, \gamma)$-stable with respect to $\mathbf \Sigma$.
    There is an algorithm $\alg_1$ taking $T$, $\eps$, $\gamma$, and $\delta \in (0,1)$ as inputs. $\alg_1$ outputs a unit vector $v \in \R^d$ such that, with probability $ \ge 1 - \delta$, $v$ is an $O(\gamma)$ $1$-ePCA of $\msig$, and $v$ lies in the span of $T$, in time
    \[O\Par{\frac{nd}{\gamma^2}\polylog\Par{\frac d {\eps\delta}}}.\]
\end{proposition}

In the context of our reduction in Theorem~\ref{thm:epca_reduce}, we will repeatedly invoke the above algorithm $\alg$ on deflated data, i.e.\ after projecting out all reported principal components so far. Conveniently, 
stability is preserved under arbitrary deflations, as shown below.
\begin{lemma}\label{lem:stable_deflation}
    Let $S \subset \R^d$ be such that the uniform distribution over $S$ is $(\eps,\gamma)$-stable with respect to $\mathbf{\Sigma}$.
    For any $\mproj \in \R^{d \times d}$, the uniform distribution over $\{\mproj x\}_{x \in S}$ is $(\eps,\gamma)$-stable with respect to $\mproj \mathbf{\Sigma} \mproj$.
\end{lemma}
\begin{proof}
For $\ma, \mb \in \Sym^{d \times d}$,  $\mathbf{A} \preceq \mathbf B$ implies $\mproj \ma \mproj \preceq \mproj \mb \mproj$ for all $\mproj \in \R^{d \times d}$, giving the claim.
\end{proof}
Thus, \Cref{prop:robust-ePCA-stability} is compatible with our framework in Theorem~\ref{thm:epca_reduce} and can be used as a $1$-ePCA oracle for $\mathbf{\Sigma}$. 
Combining this observation with \Cref{thm:epca_reduce}, we obtain nearly-linear time algorithms for robust $k$-ePCA.  Importantly, the sample complexity of our robust $k$-ePCA algorithm does not increase with $k$, as we can reuse the same samples in each call to \Cref{prop:robust-ePCA-stability} due to Lemma~\ref{lem:stable_deflation}.

\begin{corollary}[$k$-ePCA in nearly-linear time under stability] 
\label{cor:k-epca-stability}
In the setting of Proposition~\ref{prop:robust-ePCA-stability}, 
there is an algorithm $\alg_k$ taking $T$, $\eps$, $\gamma$, $\delta \in (0,1)$, and $k \in [d]$ as inputs. $\alg_k$ outputs orthonormal $\mathbf U \in \R^{d \times k}$, such that,  with probability $\ge 1 - \delta$, $\mathbf{U}$ is an $O(\gamma)$-$k$-ePCA of $\mathbf{\Sigma}$, in time 
\[O\Par{\frac{ndk}{\gamma^2}\polylog\Par{\frac d {\eps\delta}}}.\]
\end{corollary}

\Cref{cor:k-epca-stability} also leads to $k$-cPCA guarantees using \Cref{lem:etoc} and \Cref{thm:final_composition} in some parameter regimes, but we focus on ePCA for compatibility with \Cref{prop:robust-ePCA-stability}.
In Sections~\ref{ssec:subg_stability} and~\ref{ssec:ht_stability}, we apply Corollary~\ref{cor:k-epca-stability} to two fundamental distribution families: sub-Gaussian distributions and hypercontractive distributions.
Our results follow via establishing appropriate stability conditions.

\subsubsection{Sub-Gaussian distributions}\label{ssec:subg_stability}

In this section, we show that by leveraging known stability conditions for sub-Gaussian distributions, we can directly apply Corollary~\ref{cor:k-epca-stability} to design a robust $k$-PCA algorithm. We first give a definition of the sub-Gaussian distribution family under consideration.

\begin{definition}[Sub-Gaussian distribution]
    We say a distribution $\dist$ on $\R^d$ is \emph{$r$-sub-Gaussian} for a parameter $r\geq 1$ if it has mean $\0_d$ and covariance $\mathbf \Sigma$,\footnote{The same symmetrization strategy as discussed in Section~\ref{ssec:app_prelims} shows that the assumption that $\dist$ is mean-zero is without loss of generality. We defer additional discussion to \cite{JamLT20}.} and for all unit vectors $v \in \R^d$ and $t\in \R $,
    \[\E_{X \sim D}\Brack{\exp(t v^\top X)} \leq \exp\Par{\frac{t^2 r}{2} v^\top \mathbf \Sigma v}.\]
\end{definition}

Next, we present a result from \cite{JamLT20} bounding the sample complexity required for the uniform distribution over a set of corrupted samples to be stable with respect to a covariance matrix.

\begin{lemma}[{Corollary 4, \cite{JamLT20}}]\label{lem:subg_stable}
    Let $\dist$ be an $O(1)$-sub-Gaussian distribution on $\R^d$ with covariance $\mathbf \Sigma$.
    Let $\eps \in (0, \eps_0)$ for an absolute constant $\eps_0$, $\delta \in (0,1)$, and for an absolute constant $C$, let $\gamma \defeq C \eps \log(\frac 1 \eps) $. If $S$ is a set of $n$ i.i.d.\ samples from $\dist$ where, for an appropriate constant,
    \[n = \Theta\Par{\frac{d + \log(\frac 1 \delta)}{\gamma^2}},\]
    then with probability $\ge 1 - \delta$, the uniform distribution over $S$ is $(\eps,\gamma)$ stable with respect to $\mathbf \Sigma$.
\end{lemma}

Applying \Cref{prop:robust-ePCA-stability} with Lemma~\ref{lem:subg_stable} in hand, \cite[Theorem 1.2]{DKPP23} obtained a nearly-linear time algorithm for robust 1-ePCA.\footnote{\Cref{prop:robust-ePCA-stability} is applicable because if the uniform distribution over a set $S$ is $(\eps,\gamma)$-stable, then the uniform distribution over $S' \subset S$ with $|S'| \geq (1-\eps) |S|$ is also $(\eps, O(\gamma))$-stable; we apply this with $S' = S \cap T$.}
By simply replacing the use of \Cref{prop:robust-ePCA-stability} with \Cref{cor:k-epca-stability}, we further obtain the following robust $k$-PCA result for sub-Gaussian distributions.

\restaterobsubg*

As discussed previously, the sample complexity of Theorem~\ref{thm:robust_subg} notably does not grow as $k$ increases.

\subsubsection{Hypercontractive distributions}\label{ssec:ht_stability}

In this section, we relax the sub-Gaussianity assumption used in \cite{JamLT20, DKPP23} and instead assume that the data is drawn from a $(p,C_p)$-hypercontractive distribution (Definition \ref{def:hypercontractivity}).
Our main result in this section is the following algorithm for $k$-ePCA in the hypercontractive setting.

\restaterobht*
We mention that the approximation factor of $\gamma = \Theta(C_p^2 \eps^{1 - \frac{2}{p}})$ achieved by Theorem~\ref{thm:robust-k-epca-heavy-tailed} is optimal under $\eps$-corruption for a $(p,C_p)$-hypercontractive distribution (cf.\ \Cref{lem:lower-bound-robust}).
Further, the sample complexity of Theorem~\ref{thm:robust-k-epca-heavy-tailed} differs from the information-theoretic sample complexity by a factor of $\beta = C_p^6 \eps^{-\frac{2}{p}}$ and a $\log d$ term~\cite[Proposition 6.7]{LiuKO22}. 
In particular, for constant $C_p, \eps$, the sample complexity of Theorem~\ref{thm:robust-k-epca-heavy-tailed} nearly-matches that of Gaussian data even for heavy-tailed distributions, while being robust to a constant fraction of corruption and running in nearly-linear time.

At the end of the section, we give a detailed comparison between Theorem~\ref{thm:robust-k-epca-heavy-tailed} and the closest related work of \cite{KonSKO20}, which handled distributions with bounded fourth moments under an assumption related to (but not directly comparable to) our assumption of $(4, C_4)$-hypercontractivity.

Theorem~\ref{thm:robust-k-epca-heavy-tailed} follows by establishing that a large enough set of samples from hypercontractive distributions satisfy the conditions of \Cref{cor:k-epca-stability}. To establish stability, one could hope that the samples drawn from $\dist$ are stable as is (without removing any $\eps$-fraction).
    A common approach towards this goal would be using matrix Chernoff bounds, which would lead to a multiplicative dependence on $\log(\frac 1 \delta)$, where $\delta$ is the failure probability.
    Instead, we shall use the following result from \cite{DiaKP20}, phrased in our notation, that does a more white-box analysis and uses the flexibility permitted by removing $O(\eps)$-fraction of samples to establish a sharper dependence on $\log(\frac 1 \delta)$.

\begin{lemma}[{Lemma 4.1 and 4.2, \cite{DiaKP20}}]
    \label{lem:stability-heavy-tailed-dkp}
        Let $\mathcal{D}'$ be a distribution on $\R^d$ satisfying uncentered hypercontractivity, i.e.\ for an even integer $p \ge 4$, for all $v \in \R^d$,
        $(\E_{x \sim \mathcal D'}[\inprod{x}{v}^p])^{\frac 1 p} \leq \sigma_p (\E_{x \sim \mathcal D'}[\inprod{x}{v}^2])^{\half}$. 
        Let $\eps \in (0,\half)$ and $\delta \in (0,1)$.
        Suppose $\mathcal{D}'$ is supported on a ball of radius $\sqrt{R} \ge 1$, and $\mathbf{\Sigma}' \defeq \E_{x \sim \dist'}[xx^\top]$ satisfies $\normop{\mathbf \Sigma'} = O(1)$.
        Let $S$ be a set of $n$ i.i.d.\ samples from $\mathcal{D}'$, for
        \[n = \Theta\Par{\frac{R \log d}{\alpha^2} + \frac{\sigma_p^4 \log(\frac 1 \delta)}{\alpha^2} + \frac{\sigma_p^2 \log(\frac 1 \delta)}{\alpha \eps^{\frac 2 p}} 
+ \sqrt{\frac{d \log(\frac 1 \delta)}{\eps^2\alpha}} + \frac{\log(\frac 1 \delta)}{\eps}},\]
        for an appropriate constant, and $\alpha \in (0, 1)$. Then with probability $\ge 1 - \delta$, the uniform distribution on $S$ can be written as $(1 - \eps) P + \eps B$, where $P$ is $(\eps, \gamma)$-stable with respect to $\id_d$, for 
        \[\gamma = O\Par{\normsop{\msig' - \id_d} + \sigma_p^2\eps^{1 - \frac 2 p} + \alpha}.\]
    \end{lemma}

    We are now ready to give our sample complexity bound for hypercontractive stability.

\begin{proposition}
\label{thm:sample-compelxity-stability-heavy-tailed}
For an even integer $p \geq 4$, let $\mathcal D$ be a $(p,C_p)$-hypercontractive on $\R^d$ with mean $\0_d$ and covariance $\mathbf \Sigma$. Let $\eps \in (0,\eps_0)$ for an absolute constant $\eps_0$, $\delta \in (0,1)$, and $\gamma = \Theta(\sigma_p^2 \eps^{1 - \frac{2}{p}})$ be such that $\gamma \leq \half$.
If $S$ is a set of $n$ i.i.d.\ samples from $\mathcal D$ where, for an appropriate constant,
\[n = \Theta\Par{C_p^2\eps^{-\frac{2}{p}} \cdot \frac{d \log(d)}{\gamma^2}
+ \max(C_p^4,C_p^8\eps^{1 - \frac{2}{p}}) \cdot \frac{\log(\frac 1 \delta)}{\gamma^2}},\]
then with probability $\ge 1 - \delta$, the uniform distribution on $S$ can be written as $(1-\eps) P + \eps B$
for two distributions $P$ and $B$ such that $P$ is $(\eps,\gamma)$-stable with respect to $\mathbf{\Sigma}$.
\end{proposition}
\begin{proof}
    We first observe that if the uniform distribution over a set $S'$ satisfies stability with respect to $\mathbf{I}_d$, then $S:= \{\mathbf{\Sigma}^{\half} x\}_{x \in S'}$ satisfies stability with respect to $\mathbf{\Sigma}$.
    We let the samples in $S'$ be  drawn as $\mathbf{\Sigma}^{-\half} x$ for $x \sim \mathcal{D}$, and thus follow an isotropic distribution. Moreover, it is straightforward to see that this transformed distribution also satisfies $(p,C_p)$-hypercontractivity. 
    Thus, we assume $\dist$ is isotropic in the rest of the proof and establish stability with respect to $\id_d$.
    
    Next, we consider a clipped variant of $\mathcal D$. 
    Let $\mathcal D'$ be the distribution of $\mathcal{T}_R(x)$ for $R = (\frac{C_p^p}{\rho})^{\frac{2}{p-2}} d $ and $x \sim \mathcal{D}$, so $\dist'$ is supported on a ball of radius $\sqrt{R} \ge 1$.
        Let $\msig' \defeq \E_{x \sim \mathcal D'}[xx^\top]$.
    We shall choose $\rho = C' C_p^2 \eps^{1 - \frac{2}{p}} \leq \frac \gamma 2$ for a large enough constant $C'$, such that $\rho \leq 1$ by an upper bound on $\gamma$.
    By \Cref{cor:clipping-and-truncation}, $\normop{\mathbf{\Sigma}' - \mathbf{\Sigma}}  = \normop{\mathbf{\Sigma}' - \mathbf{I}_d} \leq \rho$ and the probability that a sample from $\mathcal D$ is clipped is $\le \frac \eps {100}$.
    In particular, by a Chernoff bound, this second conclusion shows that with probability at least $1 - \exp(-\Omega(n\eps)) \geq 1 - \frac \delta 2$,
     less than an $\frac \eps {2}$ fraction of samples will be clipped. 
    We let the clipped samples form the distribution $B$. In the sequel, we prove stability of the unclipped samples.    
    
    Recall that $\dist'$ has second moment matrix $\mathbf{\Sigma}'$ with $\normop{\msig'} \le 2$, and $\dist'$ satisfies uncentered hypercontractivity with $\sigma_p = 2 C_p$ by \Cref{cor:clipping-and-truncation}.
    We hence may apply Lemma~\ref{lem:stability-heavy-tailed-dkp} with $\eps \gets \frac \eps 2$, $\delta \gets \frac \delta 2$,  and $\alpha \gets \rho = \Theta(C_p^2 \eps^{1 - \frac{2}{p}})$, which again lies in $(0,1)$. Since $\rho \geq \normsop{\mathbf \Sigma' - \id_d}$, with probability $1 - \frac \delta 2$, $P$ is $(\frac \eps 2, \gamma)$-stable for $\gamma = O(C_p^2 \eps^{1 - \frac{2}{p}})$. Finally, we bound the sample complexity:
    \begingroup
    \allowdisplaybreaks
    \begin{align*}
      n &= O\Par{\frac{R \log d}{\alpha^2} + \frac{\sigma_p^4 \log(\frac 1 \delta)}{\alpha^2} + \frac{\sigma_p^2 \log(\frac 1 \delta)}{\alpha \eps^{\frac 2 p}} 
+ \sqrt{\frac{d \log(\frac 1 \delta)}{\eps^2\alpha}} + \frac{\log(\frac 1 \delta)}{\eps}}\\
      &= O\Par{\left(\frac{C_p^p}{\rho}\right)^{\frac{2}{p-2}} \frac{d \log d}{\gamma^2}  + \frac{C_p^4\log(\frac 1 \delta)}{\gamma^2} + \frac{\log(\frac 1 \delta)}{\gamma\eps^{\frac 2 p}} +  \sqrt{ \frac{d}{\gamma^2} \frac{\log(\frac 1 \delta)}{ (\eps^2/\gamma) } } + \frac{\log(\frac 1 \delta)}{\eps}} \tag*{(using the value of $R$)}\\
      &= O\Par{ C_p^2 \eps^{-\frac{2}{p}}  \frac{d \log d}{\gamma^2}  + \frac{C_p^4\log(\frac 1 \delta)}{\gamma^2} + \frac{\log(\frac 1 \delta)}{C_p^2\eps} +  \sqrt{ \frac{d}{\gamma^2} \frac{\log(\frac 1 \delta)}{ (\eps^2/\gamma) } } + \frac{\log(\frac 1 \delta)}{\eps}}\tag*{(using the value of $\gamma, \rho$)}\\
      &= O\Par{C_p^2 \eps^{-\frac{2}{p}}  \frac{d \log d}{\gamma^2}  + \frac{C_p^4\log(\frac 1 \delta)}{ \min(\gamma^2,C_p^2\eps)} + C_p^2 \eps^{-\frac{2}{p}}  \frac{d \log d}{\gamma^2} +  \frac{\log(\frac 1 \delta)}{ C_p^2 \eps^{-\frac{2}{p}}(\eps^2/\gamma) }} \tag*{(using $\sqrt{ab}\leq 2a + 2b$)}\\
       &= O\Par{C_p^2 \eps^{-\frac{2}{p}}  \frac{d \log d}{\gamma^2} + \frac{C_p^4 \log(\frac 1 \delta)}{\min(\gamma^2,C_p^2\eps, C_p^2 \eps^{-\frac{2}{p}}(\eps^2/\gamma))}} \\
       &= O\Par{C_p^2\eps^{-\frac{2}{p}}\frac{d \log d}{\gamma^2} +  C_p^4\frac{\log(\frac 1 \delta)}{\min(\gamma^2,\eps)}}\tag*{(using the value of $\gamma$)}\\
       &= O\Par{C_p^2\eps^{-\frac{2}{p}}\frac{d \log d}{\gamma^2} + \max(C_p^4, C_p^8 \eps^{1 - \frac{2}{p}} )  \frac{\log(\frac 1 \delta)}{\gamma^2}}\tag*{(using the value of $\gamma$)}.
    \end{align*}
    \endgroup
    By a union bound, with probability $1 - \delta$,
    the uniform distribution on $S$ can be written as $(1-\eps) P + \eps B$, where $P$ is $(\frac \eps 2,\gamma)$-stable.
    Finally, $(\frac \eps 2,\gamma)$-stability of $P$ implies $(\eps,O(\gamma))$ stability of $P$.\footnote{To see this, note that stability implies every second moment matrix of a $\eps$-weighted subset of the data is bounded in the Loewner order by $O(\gamma) \msig$, a bound which can be applied twice.}
\end{proof}

We now provide the proof of \Cref{thm:robust-k-epca-heavy-tailed}, by leveraging Proposition~\ref{thm:sample-compelxity-stability-heavy-tailed} and Corollary~\ref{cor:k-epca-stability}.

\begin{proof}[Proof of \Cref{thm:robust-k-epca-heavy-tailed}]

Let $S$ be the set of i.i.d.\ samples from $\mathcal{D}$. \Cref{thm:sample-compelxity-stability-heavy-tailed} implies that, with probability $\ge 1 - \delta$, the uniform distribution on $S$ can be written as $(1-\eps)P + \eps B$, where $P$ is $(\eps, \gamma)$ stable with respect to $\mathbf{\Sigma}$.
Since $\eps$ is small enough, the $(\eps,\gamma)$-stability of $P$ implies $P$ is also $(10\eps,O(\gamma))$-stable.
Let $T$ be the corrupted set of samples and observe that the uniform distribution on $T$ is $\le 2\eps$-far from $P$ in total variation.
Thus, the uniform distribution on $T$ can be written as $(1-2 \eps)P' + 2\eps B'$, where $P'$ conditions $P$ on a probability $(1 - O(\eps))$ event and consequently inherits $(8\eps,\gamma)$-stability with respect to $\mathbf{\Sigma}$ from $P$; see, for example, \cite[Lemma 2.12]{DiaKPP22-streaming}.
Reparameterizing $\eps$ to $2\eps$, we see that \Cref{cor:k-epca-stability} is applicable.
\end{proof}

We now show that the error guarantee of \Cref{thm:robust-k-epca-heavy-tailed} is optimal (up to constants) for small $\eps$ and $k=1$.
\begin{lemma}[Lower bound for robust hypercontractive ePCA]
\label{lem:lower-bound-robust}
Let $ c > 0$ be a small enough constant and let $d \in \N$ be sufficiently large. 
Let $\mathcal D$ be a $(p,C_p)$ hypercontractive distribution on $\R^d$
for an even integer $p \ge 4$ and $C_p > 2$, with mean $\0_d$ and covariance $\msig$.
Let $\eps \in (0,\eps_0)$ for $\eps_0 = (\frac{C_p^2}{8})^{\frac p 2}$ and let $\gamma = C_p^2 \eps^{1 - \frac 2 p}$.
There is no algorithm that, with probability $\ge \half$, outputs a unit vector $u$ that is a $c\gamma$-$1$-ePCA of $\msig$, given infinite samples and runtime.
\end{lemma}
\begin{proof}
    We exhibit $d+1$ mean-$\0_d$, $(p, C_p)$-hypercontractive distributions $\mathcal P = \{\dist_i\}_{0 \le i \le d}$.
    Moreover, all of the distributions $\{\dist_i\}_{i \in [d]} $ will have total variation distance $\le \eps$ from $\dist_0$.
    Finally, letting $\{\msig_i\}_{0 \le i \le d}$ be the corresponding covariances, our construction will have $\msig_0 = \id_d$, and $\msig_i = \mathbf{I}_d + s e_ie_i^\top$ for $s := \eps(0.25C_p^2\eps^{-\frac{2}{p}} - 1) = \Theta(C_p^2\eps^{1 - \frac{2}{p}}) = \Theta(\gamma)$, where we used $\eps \leq \eps_0$. 

    We now give our construction. 
    Let $P$ be the uniform distribution on $\{-1,1\}$ and let $P'$ be the distribution $(1-\eps) + \eps B$, where $B$ is the uniform distribution on $\{\pm 0.5C_p\eps^{-\frac{1}{p}}\}$. Observe that both $P$ and $P'$ are mean-zero and $(p,C_p)$-hypercontractive; indeed, for $X \sim P'$, $\E[X^2] = 1-\eps + 0.25\eps C_p^2\eps^{- \frac{2}{p}} = 1 + s$, while 
    \[\E[X^p] = (1-\eps) + \eps 2^{-p}C_p^p \eps^{-1} \leq C_p^p (\E[X^2])^{\frac p 2},\]
    where we use that $C_p \geq 2$ and $\E[X^2] \geq 1$.
    
    Define $\mathcal{D}_0$ to be the distribution on $\R^d$ such that coordinate is sampled independently from $P$.
    For each $i \in [d]$, define $\mathcal{D}_i$ to be the distribution on $\R^d$ such that all coordinates are sampled independently, the $i^{\text{th}}$ coordinate is $\sim P'$, and all other coordinates are $\sim P$.
    Since all $\mathcal{D}_i$ are product distributions, the arguments of \cite[Lemma 5.9]{KotSte17} imply that all $\mathcal{D}_i$ are also $(p,C_p)$-hypercontractive, and have the claimed covariances.
    Finally, the coupling formulation of total variation distance implies all $\{\dist_i\}_{i \in [d]}$ are at most $\eps$ far from $\mathcal{D}_0$ in the total variation distance.

    Thus, under these conditions, it is possible that the distribution of inliers is $\mathcal{D} = \mathcal{D}_{i^*}$ for some unknown $i^* \in [d]$, but the adversary corrupts the distribution to be $\mathcal{D}_0$.
    If the algorithm outputs a unit vector $\widehat{u}$ in this case, then it is an $\rho$-1-ePCA for $\mathcal{D}_{i^*}$ for $\rho = 1 - \frac{1 + s u_{i^*}^2}{1 + s} =  \frac{s (1 - u_{i^*}^2)}{1 + s}$.
    Importantly, the error $\rho \geq \frac{\min(1,s)}{4}$ if $u_{i^*}^2 \leq 1/4$.
    Since $\mathcal{D}_0$ contains no information about the unknown index $i^*$, we conclude that no algorithm can output a unit vector $u \in \R^d$ with $u_{i^*}^2 \geq \half$, for large enough $d$, and thus the ePCA approximation error is at least $c \gamma$ for a small constant $c$.
\end{proof}

\paragraph{Comparison to \cite{KonSKO20}.} As mentioned previously, the most comparable result to Theorem~\ref{thm:robust-k-epca-heavy-tailed} in the literature is due to Proposition 2.6 of \cite{KonSKO20}, which used a different set of distributional assumptions than hypercontractivity. The assumptions used in Proposition 2.6 of \cite{KonSKO20} are:\footnote{Proposition 2.6 of \cite{KonSKO20} assumes a stronger version of the first line in \eqref{eq:hypercontractivity_metalearn} where the inner product is taken against arbitrary rank-$k$ matrices with Frobenius norm $1$. However, inspection of their proof shows they only use the weaker \eqref{eq:hypercontractivity_metalearn}, i.e.\ bounded inner products against matrices of the form $\frac 1 {\sqrt k} \mmu\mmu^\top$, which they term ``semi-orthogonal.''}
\begin{equation}\label{eq:hypercontractivity_metalearn}
\begin{gathered}
    \frac 1 k \E_{x \sim \mathcal{D}}\bbb{\inner{\mmu\mmu^\top}{xx^{\top}-\msig}^{2}} \leq \nu(k)^{2} \text{ for all orthonormal } \mmu \in \R^{d \times k}, \\
    \text{and } \normop{xx^{\top}-\msig} \leq B \text{ with probability 1}. 
\end{gathered}
\end{equation}
For comparison, if $\dist$ is a $(4, C_4)$-hypercontractive distribution, 
\ba{
    \E_{x \sim \mathcal{D}}\bbb{\inner{vv^{\top}}{xx^{\top}}^{2}} \leq C_{4}^{2}\inner{vv^{\top}}{\msig}^{2} \text{ for all } v \in \R^d.
    \label{eq:hypercontractivity_ours}
}
The assumption \eqref{eq:hypercontractivity_metalearn} differs in that applies to rank-$k$ matrices (rather than rank-$1$), it enforces a uniform bound $\nu(k)$ (rather than a bound sensitive to $v$, which is tighter if the spectrum of $\msig$ is uneven), and enforces an almost sure bound, which does not follow in general for hypercontractive distributions. We now bound the parameters in \eqref{eq:hypercontractivity_metalearn} under the assumption \eqref{eq:hypercontractivity_ours} for fair comparison.

Let $\mmu \in \R^{d \times k}$ with columns $\{u_i\}_{i \in [k]}$ be orthonormal, and suppose that \eqref{eq:hypercontractivity_ours} holds. Then,
\begin{equation*}
\begin{aligned}
\E_{x \sim \mathcal{D}}\bbb{\inner{\mmu\mmu^{\top}}{xx^{\top}-\msig}^{2}} &= \E_{x \sim \mathcal{D}}\bbb{\inner{\mmu\mmu^{\top}}{xx^{\top}}^{2}-2\inner{\mmu\mmu^{\top}}{xx^{\top}}\inner{\mmu\mmu^{\top}}{\msig} + \inner{\mmu\mmu^{\top}}{\msig}^{2}} \\
&= \E_{x \sim \mathcal{D}}\bbb{\inner{\sum_{i\in[k]} u_iu_i^\top}{xx^{\top}}^{2}} - \inner{\mmu\mmu^{\top}}{\msig}^{2} \\
&= \sum_{i\in[k]} \sum_{j\in[k]} \E_{x \sim \mathcal{D}}\bbb{\inner{u_i}{x_i}^2 \inner{u_j}{x_j}^2} - \inner{\mmu\mmu^{\top}}{\msig}^{2} \\
&\le \sum_{i\in[k]} \sum_{j\in[k]} \sqrt{\E_{x \sim \mathcal{D}}\bbb{\inner{u_i}{x_i}^4}\E_{x \sim \mathcal{D}}\bbb{\inner{u_j}{x_j}^4}}  - \inner{\mmu\mmu^{\top}}{\msig}^{2} \\
&= \bb{\sum_{i\in[k]} \sqrt{\E_{x \sim \mathcal{D}}\bbb{\inner{u_i}{x_i}^4}}}^2  - \inner{\mmu\mmu^{\top}}{\msig}^{2} \\
&\le \bb{\sum_{i\in [k]} C_4 \inner{u_iu_i^\top}{\msig}}^2 - \inner{\mmu\mmu^{\top}}{\msig}^{2} \\
&= \bb{C_4^2-1}\inner{\mmu\mmu^\top}{\msig}^2 \le \Par{C_4^2 - 1}\norm{\msig}_k^2.
\end{aligned}
\end{equation*}
The first inequality used the Cauchy-Schwarz inequality and the second inequality used \eqref{eq:hypercontractivity_ours}. It follows that
the value of $\nu(k)$ that can be derived in~\eqref{eq:hypercontractivity_metalearn} satisfies 
\[\nu(k) \le \sqrt{C_4^2 - 1} \cdot \frac{\norm{\msig}_k}{\sqrt k}.\]
This bound is tight up to constant factors, as witnessed by the hypercontractive mixture distribution $\half \Nor(\0_d, \half \id_d) + \half \Nor(\0_d, \frac 3 2 \id_d)$, whose covariance matrix is $\msig = \id_d$ and whose hypercontractive parameter in \eqref{eq:hypercontractivity_ours} is a constant bounded away from $1$, since the fourth moment of a standard Gaussian is $3$. In this case, we claim that for any orthonormal matrix $\mmu \in \R^{d \times k}$, $\inprod{\mmu\mmu^\top}{xx^\top}$ deviates from its expectation by $\Omega(k)$ with constant probability, which is enough to lower bound $\nu(k)$ by $\Omega(\sqrt k)$ by using the definition \eqref{eq:hypercontractivity_metalearn}, since the left-hand side is $\frac 1 k$ times the variance of $\inprod{\mmu\mmu^\top}{xx^\top}$. The expectation of $\inprod{\mmu\mmu^\top}{xx^\top}$ is simply $k$. On the other hand, if $x \sim \Nor(\0_d, \frac 3 2 \id_d)$ (which happens with probability $\half$), the Hanson-Wright inequality implies that with high constant probability, we have $|x^\top \mmu\mmu^\top x - \frac 3 2 k| \le \frac 1 3 k$, so $x^\top \mmu\mmu^\top x \ge \frac 7 6 k$. Therefore, the variance of $\inprod{\mmu\mmu^\top}{xx^\top}$ is indeed $\Omega(k^2)$, so $\nu(k) = \Omega(\sqrt k)$, matching the above inequality up to a constant factor.

Next, by Proposition 2.6 of \cite{KonSKO20}, the output $\widetilde{\mmu}$ of their proposed algorithm satisfies
\bas{
\norm{\msig}_k - \inner{\widetilde{\mmu}\widetilde{\mmu}^\top}{\msig} = O\bb{\epsilon \norm{\msig}_k +\nu(k) \sqrt{k\epsilon}} = O\bb{\gamma}\norm{\msig}_k,
}
where $\gamma = \sqrt{\epsilon}$, which is implied by the upper bound in Theorem~\ref{thm:robust-k-epca-heavy-tailed} up to constant factors for any $C_4 = \Theta(1)$. Next, using the clipping argument in Lemma~\ref{lemma:clippingclose} for $B := C_{4}^{2} \cdot \frac{\Tr\msig }{\epsilon}$, under $\eqref{eq:hypercontractivity_ours}$, $\Pr(\normop{xx^{\top}} \geq B) \leq \epsilon$.\footnote{This is the threshold at which point the clipped samples can be treated as outliers.}
Therefore, the sample complexity used by Proposition 2.6 of \cite{KonSKO20} is 
\bas{
    n &= O\bb{\bb{dk^{2} + \frac{B}{\nu\bb{k}}\sqrt{k\epsilon}}\frac{\log\bb{\frac{d}{\delta\epsilon}}}{\epsilon}} \\
    &= O\bb{\bb{\frac{dk^{2}}{\epsilon} + \frac{\Tr\bb{\msig}}{\norm{\msig}_{k}} \frac{k}{\epsilon^{1.5}}}\log\bb{\frac{d}{\delta\epsilon}}},
}
for $C_4 = \Theta(1)$, which is worse than the sample complexity required in Theorem~\ref{thm:robust-k-epca-heavy-tailed} in the dependence on $k$. For example, in the reasonably well-conditioned regime where $\Tr(\msig) = \Theta(\frac d k \cdot \norms{\msig}_k)$, the respective sample complexities of \cite{KonSKO20} and our algorithm for the same estimation rate are
\[O\Par{\Par{\frac{dk^2}{\eps} + \frac{d}{\eps^{1.5}}} \cdot \log\Par{\frac d {\delta\eps}}},\;\;\; O\Par{\frac{d\log d + \log\frac 1 \delta}{\eps^{1.5}}}. \]

\subsection{Online heavy-tailed PCA}\label{ssec:htoja}

In this section, we provide a $k$-cPCA algorithm for heavy-tailed data without adversarial corruptions in an online setting, as a proof-of-concept application of our cPCA reduction. More precisely, we let $X_{1}, X_{2}, \ldots X_{n} \sim \mathcal{D}$ be i.i.d.\ draws from a $\bb{p,C_{p}}$-hypercontractive distribution, $\dist$, with covariance matrix $\mathbf{\Sigma}$. Our goal is to perform an approximate $k$-cPCA of $\msig$, in the setting where the samples arrive online, i.e.\ we are limited to using $O(kd)$ space, where $k \ll d$ but $n$ is potentially $\gg d$. 

We first establish a helper result, Lemma~\ref{lemma:clipping_cpca_bound}, which shows that an approximate $k$-cPCA of $\mathbf{\Sigma} \in \PSD^{d \times d}$ is also an approximate $k$-cPCA of $\hmsig \in \PSD^{d \times d}$ provided $\mathbf{\Sigma}$ and $\hmsig$ are sufficiently close in operator norm. Our result follows from standard eigenvalue perturbation bounds from the literature.

\begin{lemma}[cPCA perturbation]
    \label{lemma:clipping_cpca_bound}
Let $\mathbf{\Sigma}, \widehat{\mathbf{\Sigma}} \in \PSD^{d \times d}$ satisfy $\normsop{\mathbf{\Sigma} - \widehat{\mathbf{\Sigma}}} \leq \rho$. Let $\mmu \in \R^{d \times k}$ be a $(\delta, \gamma)$-$k$-cPCA of $\widehat{\mathbf{\Sigma}}$ for $\max\left\{\delta, \gamma\right\} \leq \frac{1}{10}$. Then for $\rho < \frac{\gamma \lambda_k(\mathbf{\Sigma})}{2}$, the following hold.
    \begin{enumerate}
        \item $\mmu$ is a $(\Delta, \Gamma)$-$k$-cPCA of $\mathbf{\Sigma}$ with $\Delta := 8k\Pars{\frac{\rho}{\gamma\lambda_k(\mathbf{\Sigma})}}^{2}+2\delta$, $\Gamma := 2\gamma$.
        \item $\kappa_{k}\Pars{\widehat{\msig}} \leq 2\kappa_{k}\Pars{\msig}$.
    \end{enumerate} 
\end{lemma}
\begin{proof}
Let $\widehat{\ms} := \widehat{\mv}^{\le (1 - \gamma)\lam_{k}(\widehat{\mathbf{\Sigma}}) }(\widehat{\mathbf{\Sigma}})$ denote the eigenspace of $\widehat{\mathbf{\Sigma}}$ with eigenvalues  $\le (1 - \gamma)\lam_{k}(\widehat{\mathbf{\Sigma}})$ and let $\widehat{\ml}$ denote its complement subspace. Similarly, define $\ms := \mv^{\le (1 - \Gamma)\lam_{k}(\mathbf{\Sigma}) }(\mathbf{\Sigma})$ and let $\ml$ denote its complement subspace. By definition of $\mmu$, we have $\normsf{\hms^\top \mmu}^2 = \Tr(\mmu^\top\hms\hms^\top\mmu) \le \delta$. 
Hence,
\ba{
    \normsf{\ms^\top\mmu}^2 = \Tr\bb{\mmu^{\top}\ms\ms^{\top}\mmu} 
    &= \Tr\bb{\mmu^{\top}\bb{\hml\hml^{\top} + \hms\hms^{\top}}\ms\ms^{\top}\bb{\hml\hml^{\top} + \hms\hms^{\top}}\mmu} \notag \\
    &= \Tr\bb{\bb{\mmu^{\top}\hml\hml^{\top}\ms + \mmu^{\top}\hms\hms^{\top}\ms}\bb{\ms^{\top}\hml\hml^{\top}\mmu + \ms^{\top}\hms\hms^{\top}\mmu}} \notag \\
    &\leq 2\Tr\bb{\mmu^{\top}\hml\hml^{\top}\ms\ms^{\top}\hml\hml^{\top}\mmu} + 2\Tr\bb{\mmu^{\top}\hms\hms^{\top}\ms\ms^{\top}\hms\hms^{\top}\mmu}, \text{ using Fact}~\ref{fact:trace_cs} \notag \\
    &= 2\Tr\bb{\hml^{\top}\ms\ms^{\top}\hml\hml^{\top}\mmu\mmu^{\top}\hml}
    +
    2\Tr\bb{\ms\ms^{\top}\hms\hms^{\top}\mmu\mmu^{\top}\hms\hms^{\top}} \notag \\
    &\leq 2\normop{\hml^{\top}\ms\ms^{\top}\hml}\Tr\bb{\hml^{\top}\mmu\mmu^{\top}\hml}
    + 2\normop{\ms\ms^{\top}}\Tr\bb{\hms^{\top}\hms\hms^{\top}\mmu\mmu^{\top}\hms} \notag \\
    &\leq 2\normop{\ms^{\top}\widehat{\ml}}^2\Tr\bb{\hml\hml^{\top}\mmu\mmu^{\top}}
    + 2\normop{\ms\ms^{\top}}\Tr\bb{\hms^{\top}\mmu\mmu^{\top}\hms} \notag \\
    &\leq 2\normop{\ms^{\top}\widehat{\ml}}^2 \normop{\hml\hml^{\top}} \Tr\bb{\mmu\mmu^{\top}}
    + 2\normop{\ms\ms^{\top}}\normf{\hms^{\top}\mmu}^2 \notag \\
    &\le 2k\normop{\ms^{\top}\widehat{\ml}}^2 + 2\delta, \text{ using the cPCA guarantee.} \label{eq:clipping_cpca_bound_1} }
Further, using Fact~\ref{fact:weyl},
\ba{
    \left|\lambda_{k}\Pars{\hmsig} - \lambda_{k}\bb{\msig}\right| \leq \normop{\msig - \hmsig} \leq \rho. \label{eq:weyl_eigenvalue_bound}
}
Therefore for $\Gamma = 2\gamma$, 
\bas{
     (1 - \gamma)\lam_{k}(\mathbf{\widehat{\Sigma}}) - (1 - \Gamma)\lam_{k}(\mathbf{\Sigma})
     &= \gamma \lambda_{k}\bb{\mathbf{\Sigma}} + \bb{1-\gamma}\bb{\lam_{k}(\mathbf{\widehat{\Sigma}}) - \lam_{k}(\mathbf{\Sigma})} \\
     &\geq \gamma \lambda_{k}\bb{\mathbf{\Sigma}} - \rho\bb{1-\gamma}, \text{ using}~\eqref{eq:weyl_eigenvalue_bound} \\
     &\geq \frac{\gamma \lambda_{k}\bb{\mathbf{\Sigma}}}{2}, \text{ using } \rho \leq \frac{\gamma \lambda_{k}\bb{\mathbf{\Sigma}}}{2}.
}
Hence, by applying the gap-free Wedin theorem (Lemma B.3, \cite{allen2016lazysvd}), we have
\ba{
    \normop{\ms^{\top}\widehat{\ml}} \leq \frac{\rho}{(1 - \gamma)\lam_{k}(\widehat{\mathbf{\Sigma}}) - (1 - \Gamma)\lam_{k}(\mathbf{\Sigma}) } \le \frac{2\rho}{\gamma \lambda_k(\mathbf{\Sigma})}. \label{eq:mstop_hml_bound}
}
Combining \eqref{eq:clipping_cpca_bound_1} and \eqref{eq:mstop_hml_bound} yields the first claim, as desired:
\bas{
\normsf{\ms^\top\mmu}^2 \le \frac{8k\rho^2}{\gamma^2\lambda_k(\mathbf{\Sigma})^2}+2\delta.
}
Finally, for the claim regarding $\kappa_{k}\Pars{\widehat{\msig}}$, we have by two applications of Fact~\ref{fact:weyl}:
\bas{
    \kappa_{k}\Pars{\widehat{\msig}} &= \frac{\lambda_{1}\Pars{\widehat{\msig}}}{\lambda_{k}\Pars{\widehat{\msig}}} 
    \leq \frac{\lambda_{1}\bb{\msig} + \Abs{\lambda_{1}\bb{\msig}-\lambda_{1}\Pars{\widehat{\msig}}}}{\lambda_{k}\bb{\msig} - \Abs{\lambda_{k}\bb{\msig}-\lambda_{k}\Pars{\widehat{\msig}}}} \leq \frac{\lambda_{1}\bb{\msig} + \frac{\gamma \lambda_{k}\bb{\mathbf{\Sigma}}}{2}}{\lambda_{k}\bb{\msig} - \frac{\gamma \lambda_{k}\bb{\mathbf{\Sigma}}}{2}} \leq \frac{1+\frac{\gamma}{2}}{1-\frac{\gamma}{2}}\kappa_{k}\bb{\msig} \le 2\kappa_k(\msig).
}
\end{proof}
We specify a choice of $\rho$ for convenient application of Lemma \ref{lemma:clipping_cpca_bound} in the following.

\begin{corollary}
    \label{cor:clipping_cpca_bound} In the setting of Lemma \ref{lemma:clipping_cpca_bound}, if $\rho := \sqrt{\frac{\delta}{8k}}\gamma\lambda_{k}\bb{\msig}$, $\mmu$ is a $(3\delta, 2\gamma)$-$k$-cPCA of $\mathbf{\Sigma}$.
\end{corollary}

Corollaries~\ref{cor:clipping-and-truncation} and~\ref{cor:clipping_cpca_bound} make our roadmap clear for heavy-tailed PCA. We first use Corollary~\ref{cor:clipping-and-truncation} to devise a clipping threshold for our data such that the clipped covariance matrix is close in operator norm to the original distribution's covariance matrix. We then perform approximate $k$-cPCA on the clipped data, analyzed using our reduction in Theorem~\ref{thm:final_composition}, and finally use Corollary~\ref{cor:clipping_cpca_bound} to convert the resulting $k$-cPCA on the clipped covariance to a corresponding $k$-cPCA on the original covariance.

To demonstrate an instantiation of this strategy, we propose an online gap-free $k$-cPCA algorithm for heavy-tailed data using Oja's algorithm~\cite{oja1982simplified} as the $1$-cPCA oracle. More specifically, we use the following guarantee on Oja's algorithm from \cite{allen2017first}.

\begin{proposition}[Theorem 2, \cite{allen2017first})]
\label{prop:oja_1cpca_oracle}
There is an algorithm $\Oja$ with the following guarantee. Let $\delta, \gamma \ge 0$ such that $\max\bb{\delta, \gamma} \le \frac{1}{10}$ and $\delta \le \frac{\gamma^2}{256}$. Let $d \in \N$, $R > 0$, and $\beta \in (0,1)$. 
Let $n = \Omega(\frac{d}{\delta\gamma^{2}}\log(\frac{d}{\beta}))$, and let $\{X_i\}_{i \in [n]} \in \R^d$ be drawn i.i.d.\ from $\dist$, a distribution on $\R^d$ with mean $\0_d$ and covariance $\msig$, where $\norm{X}_2^2 \le R$ almost surely for $X \sim \dist$. Then, with probability at least $1-\beta$, $\Oja(\delta, \gamma, \{X_i\}_{i \in [n]}, R)$ returns a $(\delta, \gamma)$-cPCA of $\mm$, using $O(d)$ space.

\end{proposition}

As a consequence of Proposition~\ref{prop:oja_1cpca_oracle} and our roadmap described earlier, we conclude with the following.

\begin{theorem}\label{thm:oja_k}
For an even integer $p \ge 4$, let $\dist$ be $(p, C_p)$-hypercontractive on $\R^d$ with mean $\0_d$ and covariance $\msig$. Let $(\Delta, \Gamma) \in (0, 1)$, assume $\Delta \cdot \kappa_k(\msig)^2 \le \Gamma^2$, and set $\delta = \frac{1}{k^{\Theta(\log k)}} \cdot \Delta$, $\gamma = \frac{1}{\Theta(k^3)} \cdot \Gamma$, for appropriate constants. Let 
\[\alpha := \Par{\frac{C_{p}^{2}\kappa_{k}\bb{\msig}\sqrt{k}}{\Gamma\sqrt{\Delta}}}^{\frac{1}{p-2}},\; R \defeq \Theta(\alpha \Tr\bb{\msig}).\]
Let $\beta \in (0, 1)$. If $n = \Theta(\alpha\frac{d \kappa_{k}\bb{\msig}^{2}}{\delta\gamma^{2}}\log(\frac{d}{\beta}))$ for an appropriate constant, $\BBPCA$ (Algorithm \ref{alg:bbpca}) using $\Oja$ as a $(\delta, \gamma)$-$1$-cPCA oracle on $n$ samples from $\mathcal{T}_R(\dist)$ returns a $(\Delta, \Gamma)$-$k$-cPCA of $\msig$ with probability $\ge 1 - \beta$, in $O(dk)$ space.
\end{theorem}
\begin{proof}
Let $\rho \defeq (\frac{\Delta}{24k})^{1/2}\frac{\Gamma}{\kappa_{k}\bb{\msig}}$. Using Corollary~\ref{cor:clipping_cpca_bound}, for $R\geq \bb{\frac{C_{p}^{p}}{\rho}}^{\frac{2}{p-2}}\Tr\bb{\mathbf{\Sigma}}$, the covariance of the clipped distribution $\mathcal{T}_R(\dist)$, $\msig'$, satisfies 
    \bas{
        \normop{\msig'-\msig} \leq \rho\normop{\msig}.
    }
    Next, we obtain a $k$-cPCA, $\mmu$, of $\msig'$ by running Algorithm \ref{alg:bbpca} on the clipped distribution. Using Theorem~\ref{thm:final_composition} and the guarantees of Proposition~\ref{prop:oja_1cpca_oracle}, we have that $\mmu$ is a $\bb{\frac{\Delta}{3},\frac{\Gamma}{2}}$-$k$-cPCA of $\msig'$. Finally, we claim $\mmu$ is also a $\bb{\Delta, \Gamma}$-$k$-cPCA of $\msig$ by Corollary~\ref{cor:clipping_cpca_bound}, which requires
    \ba{
        \normop{\msig'-\msig} \leq \rho\normop{\msig} \leq 2\sqrt{\frac{\Delta}{24k}}\Gamma\lambda_{k}\bb{\msig}. \label{eq:norm_bound_msig'}
    }
    The proof follows by substituting the definition of $\rho$ in \eqref{eq:norm_bound_msig'}. The space complexity is immediate from the definition of Algorithm~\ref{alg:bbpca}.
\end{proof}

We do note that the prior work \cite{allen2017first}, in addition to proving Proposition~\ref{prop:oja_1cpca_oracle}, gave a sophisticated analysis of a simultaneous variant of Oja's algorithm using a $d \times k$ block matrix, which applies to distributions with almost surely bounded supports (with rates parameterized by the covariance matrix and the almost sure bound). While this \cite{allen2017first} result does not directly apply to hypercontractive distributions, following the same roadmap as in Theorem~\ref{thm:oja_k}, i.e.\ first truncating the distribution (via Corollary~\ref{cor:clipping-and-truncation}) and then bounding the perturbation (via Corollary~\ref{cor:clipping_cpca_bound}), but using the $k$-Oja result of \cite{allen2017first} in place of Proposition~\ref{prop:oja_1cpca_oracle} and Theorem~\ref{thm:final_composition}, obtains an improvement over Theorem~\ref{thm:oja_k}. For instance, \cite{allen2017first}'s $k$-Oja analysis applies in all parameter regimes because it is not bottlenecked by the impossibility result in Proposition~\ref{prop:delta_gg_gamsquare_bad}, and also incurs only a polynomial overhead in $k$.

We include Theorem~\ref{thm:oja_k} as a proof-of-concept of how to apply our reduction to give a more straightforward $k$-cPCA result by relying only on existence of the corresponding $1$-cPCA algorithm, rather than designing a custom analysis as was done in \cite{allen2017first}. We are optimistic about the utility of the approach in Theorem~\ref{thm:oja_k} in providing tools for attacking future statistical PCA settings with various constraints (e.g.\ privacy, dependent data, and so forth as mentioned in Section~\ref{sec:intro}) via reductions, as $1$-cPCA algorithms are typically more straightforward to analyze than $k$-cPCA algorithms.

\newpage

\bibliographystyle{alpha}
\bibliography{refs}

\newpage

\begin{appendix}
\end{appendix}

\end{document}